\documentclass [11pt,oneside,a4paper,mathscr]{amsart}

\usepackage{amscd,amsmath,amssymb,euscript}
\usepackage[frame,cmtip,curve,arrow,matrix,line,graph]{xy}
\usepackage{tikz-cd}
\usepackage{bigints}
\usepackage{hyperref}
\usetikzlibrary{matrix}

\title{Relative stable pairs and a non-Calabi-Yau wall crossing}
\author{Tudor P\u adurariu}
\address{Columbia University, 
Mathematics Hall, 2990 Broadway, New York, NY 10027}
\email{tgp2109@columbia.edu}

\date{}

\newtheorem{thm}{Theorem}[section]

\newtheorem{prop}[thm]{Proposition}
\newtheorem{lemma}[thm]{Lemma}
\newtheorem{conjecture}[thm]{Conjecture}

\theoremstyle{definition}
\newtheorem{defn}[thm]{Definition}

\newtheorem{thm*}[thm]{Theorem$^*$}

\newcommand{\comment}[1]{}

\renewcommand{\leq}{\leqslant}
\renewcommand{\geq}{\geqslant}

\newcommand{\dual}{\vee}

\newcommand{\F}{\mathcal{F}}

\newcommand{\OO}{\mathcal{O}}

\newcommand{\HH}{\mathcal{H}}

\newcommand{\I}{\mathcal{I}}

\newcommand{\K}{\mathcal{K}}

\newcommand{\ch}{\operatorname{ch}}

\newcommand{\C}{\mathbb{C}}

\begin{document}
\maketitle

\begin{abstract}
Let $Y$ be a smooth projective threefold and let $f:Y\to X$ be a birational map with $Rf_*\mathcal{O}_Y=\mathcal{O}_X$.
When $Y$ is Calabi-Yau, Bryan--Steinberg defined enumerative invariants associated to such maps called $f$-relative stable (or Bryan-Steinberg) invariants. When $X$ has Gorenstein singularities and $f$ has relative dimension one, they compared these invariants to the Donaldson-Thomas, or equivalently Pandharipande-Thomas invariants of $Y$.

We define Bryan--Steinberg invariants for maps $f$ as above without assuming that $Y$ is Calabi-Yau. For $X$ with Gorenstein and rational singularities, $f$ of relative dimension one, and for insertions from $X$ and arbitrary descendant levels, we conjecture a relation between the generating functions of Bryan-Steinberg and Pandharipande-Thomas invariants of $Y$.
 We check the conjecture for 
the contraction $f: Y\to X$ of a rational curve $C$ with normal bundle $N_{C/Y}\cong \mathcal{O}_C(-1)^{\oplus 2}$ using degeneration and localization techniques to reduce to a Calabi-Yau situation, which we then treat using Joyce's motivic Hall algebra.
\end{abstract}

\section{Introduction}

\subsection{Bryan-Steinberg pairs in the Calabi-Yau case}

Let $Y$ be a smooth projective threefold and let $f:Y\to X$ be a birational map with $Rf_*\mathcal{O}_Y=\mathcal{O}_X$. 

When $Y$ is Calabi-Yau, Bryan--Steinberg
defined invariants which ``count" certain two term complexes, which we call BS pairs; they are called $f$-relative stable pairs in the original paper \cite{bs}. When $X$ has Gorenstein and rational singularities,
Bryan--Steinberg 
showed that the generating series of BS invariants 
is
\begin{equation}\label{BSratio}
    \text{BS}^f(q)=\frac{\text{PT}^Y(q)}{\text{PT}^{Y, \text{exc}}(q)},
    \end{equation}
where $\text{PT}^Y(q)$ is the generating series of Pandharipande-Thomas (PT) invariants of $Y$ and $\text{PT}^{Y, \text{exc}}(q)$ is the generating series of PT invariants supported on exceptional curves, see Theorem \ref{wcbryan}. When $X$ as above is the coarse space of certain Calabi-Yau DM stacks of dimension $3$, 
the ratio on the right hand side of \eqref{BSratio}, or equivalently the analogous ratio of Donaldson-Thomas (DT) series \cite{b}, appears in the statement of the DT crepant resolution conjecture of Bryan--Cadman--Young \cite{bcy}. The introduction of BS pairs was a first step towards proving the DT crepant resolution conjecture by providing a geometric interpretation of the ratio of PT series in \eqref{BSratio}. The conjecture was eventually proved by Beentjes--Calabrese--Rennemo \cite{bcr} using motivic Hall algebra techniques.

\subsection{Bryan-Steinberg pairs beyond the Calabi-Yau case} 
We define BS pairs when $Y$ is not necessarily Calabi-Yau in the same way as Bryan--Steinberg.
Let $\text{Coh}_{\leq 1}(Y)$ be the abelian category of coherent sheaves on $Y$ with support of dimension at most $1$. 
Consider the torsion pair $(\textbf{T}, \textbf{F})$ of $\text{Coh}_{\leq 1}(Y)$ defined as follows: $\textbf{T}\subset \text{Coh}_{\leq 1}(Y)$ is the full subcategory of sheaves $T$ such that $$Rf_*T\in \text{Coh}_{\leq 0}(X)$$ is a sheaf on $X$ with support of dimension at most zero, and $\textbf{F}\subset \text{Coh}_{\leq 1}(Y)$ is its complement. A two-term complex \[I=\left[\OO_Y\xrightarrow{s} F\right]\] is called a \textit{Bryan-Steinberg (BS) pair} if $F\in\textbf{F}$ and 
$\text{coker}\,(s)\in\textbf{T}$. 
When $X$ is smooth and $f$ is the identity, BS pairs are the same as PT pairs. 
\smallbreak

Denote by $N_1(Y)$ the abelian group of curves in $Y$ up to numerical equivalence, by $N^{\text{exc}}_1(Y)\subset N_1(Y)$ the subspace of curves supported on the non-trivial fibers of $f$, and let $N_{\leq 1}(Y):=N_1(Y)\oplus\mathbb{Z}$ and $N^{\text{exc}}_{\leq 1}(Y):= N^{\text{exc}}_1(Y)\oplus\mathbb{Z}$.
Consider the moduli stack $\mathfrak{M}_Y(\OO)$ of sheaves on $Y$ with a section.
Bryan--Steinberg showed that the locus $$\text{BS}^f_n(Y,\beta)\subset \mathfrak{M}_Y(\OO)$$ of BS pairs $\left[\OO_Y\xrightarrow{s} F\right]$ with $\text{ch}\,(F)=(0, 0, \beta, n)\in H^{\text{even}}(Y)$ with $\beta\in H^4(Y, \mathbb{Z})$ a curve class and $n\in H^6(Y, \mathbb{Z})\cong \mathbb{Z}$ is a finite-type constructible set \cite{bs}. 
When $Y$ is Calabi-Yau, the BS invariants are defined using the Behrend function $\nu:\mathfrak{M}_Y(\OO)\to \mathbb{Z}$:
$$\text{BS}^f(Y,\beta,n):=\sum_{n\in\mathbb{Z}}n\,\chi\left(\nu^{-1}(n)\right).$$
In Section \ref{fund}, we define analogous enumerative invariants when $Y$ is not necessarily Calabi-Yau by integrating against a natural virtual fundamental class.

\begin{thm}\label{thm1}
Let $Y$ be a smooth projective threefold, let $f:Y\to X$ be a birational map with $Rf_*\mathcal{O}_Y=\mathcal{O}_X$,
and let $(\beta,n)\in N_{\leq 1}(Y)$. 
The functor $\Phi_{\text{BS}(Y)}$ of Bryan-Steinberg pairs $I=\left[\OO_Y\xrightarrow{s} F\right]$ with $\text{ch}\,(F)=(\beta, n)\in N_{\leq 1}(Y)$ is representable by a proper algebraic space $\text{BS}^f_n(Y,\beta)$ with a natural virtual fundamental class $\left[\text{BS}^f_n(Y,\beta)\right]^{\text{vir}}\in A_{d}\left(\text{BS}^f_n(Y,\beta)\right)$, where $d=-\chi\left(\text{RHom}(I,I)_0\right)$.
\end{thm}

By the Artin representability criterion, the functor $\Phi_{\text{BS}(Y)}$ is representable by a proper algebraic space if $\Phi_{\text{BS}(Y)}$ is open (Proposition \ref{open}), bounded (already proved by Bryan--Steinberg), separated (Proposition \ref{sep}), complete (Proposition \ref{com}), and has trivial automorphisms (Proposition \ref{aut}).
The proof of separatedness and completeness is similar to Langton's argument that the moduli of semistable sheaves on a projective variety has the same properties \cite{l}. 

By work of Beentjes--Calabrese--Rennemo \cite{bcr}, BS pairs are open in the derived category of gluable complexes on $Y$.
The existence of a virtual fundamental class for $\text{BS}^f_n(Y,\beta)$ now follows from the work of Huybrechts--Thomas \cite{ht}.

\subsection{Wall-crossing between BS and PT invariants}

When $Y$ is Calabi-Yau, consider the generating series of BS invariants \cite{bs}:
$$\text{BS}^f(q):=\sum_{(\beta, n)\in N_{\leq 1}(Y)}\text{BS}^f(Y,\beta, n)\,q^{\beta+n}.$$ 
The generating series $\text{PT}^Y(q)$ of PT invariants is defined similarly. We also consider the generating series $$\text{PT}^{Y, \text{exc}}(q):=\sum_{(\beta, n)\in N^{\text{exc}}_{\leq 1}(Y)}\text{PT}(Y,\beta, n)\,q^{\beta+n}.$$ 
Both generating series are elements of the Laurent ring $\mathbb{C}[\Delta]_{\Phi}$, see Definition \ref{Laurent}.
Bryan--Steinberg proved the following wall-crossing theorem using identities in the motivic Hall algebra of $Y$:

\begin{thm}\label{wcbryan}
Let $X$ be a projective threefold with Gorenstein and rational singularities and let $f: Y\to X$ be a resolution of singularities of relative dimension $1$ with $Y$ a projective Calabi-Yau. The generating series for BS and PT invariants are related by
$$\text{BS}^f(q)=\frac{\text{PT}^Y(q)}{\text{PT}^{Y, \text{exc}}(q)}.$$
\end{thm}

For $Y$ not necessarily Calabi-Yau, we can define BS and PT invariants with insertions, see Subsections \ref{inse} and \ref{pt}, respectively. 
Let $(\beta, n)\in N_{\leq 1}(Y)$ and consider insertions $\gamma_1,\cdots,\gamma_r\in f^*H^{\cdot}(X,\mathbb{Z})$ and descendent levels $\kappa_1,\cdots, \kappa_r\geq 0$. We define BS invariants with insertions $\langle \tau_{\kappa_1}(\gamma_1),\cdots, \tau_{\kappa_r}(\gamma_r)\rangle_{\beta,n}$, see Subsection \ref{inse}. Consider the generating series of BS invariants with insertions $\gamma_1,\cdots, \gamma_r$ and descendant levels $\kappa_1,\cdots, \kappa_r\geq 0$:
$$\text{BS}^f(q; \gamma, \kappa):=\sum_{(\beta, n)\in N_{\leq 1}(Y)} \langle \tau_{\kappa_1}(\gamma_1),\cdots, \tau_{\kappa_r}(\gamma_r)\rangle_{\beta,n}\,q^{\beta+n}.$$ 
The definitions of the generating series $\text{PT}^Y(q; \gamma, \kappa)$ and $\text{PT}^{Y, \text{exc}}(q)=\text{PT}^{Y, \text{exc}}(q; \emptyset, \emptyset)$ are similar. 

\begin{conjecture}\label{conj}
Let $f:Y\to X$ be as in Theorem \ref{wcbryan}, but $Y$ no longer necessarily Calabi-Yau, let $E\subset Y$ be the exceptional locus, and let $Z:=f(E)\subset X$. 
Consider insertions $\gamma_1,\cdots, \gamma_r\in f^*H^{\geq \dim Z+1}(X,\mathbb{Z})$ and arbitrary descendant levels $\kappa_1,\cdots, \kappa_r\geq 0$. The generating series for BS and PT invariants with these corresponding insertions and descendant levels are related by
$$\text{BS}^f(q; \gamma, \kappa)=\frac{\text{PT}^{Y}(q; \gamma, \kappa)}{\text{PT}^{Y, \text{exc}}(q)}.$$
\end{conjecture}

 The conjecture above is similar to the DT/PT correspondence conjectured by Pandharipande--Thomas \cite[Conjecture 3.28]{pt} and to the DT crepant resolution conjecture of Bryan--Cadman--Young \cite{bcy}. The DT/PT correspondence was proved in the Calabi-Yau case by Bridgeland \cite{b}, Toda \cite{t1} using the machinery of motivic Hall algebras developed by Joyce \cite{j1}, \cite{j2}, \cite{j3}, and Joyce--Song \cite{js}. The DT/PT correspondence is not known in non-Calabi-Yau cases. 

When $Y$ is Calabi-Yau, the definition of BS invariants using the virtual fundamental class from Theorem \ref{thm1} is the same as the one using the Behrend function by the main result in \cite{be}, so
Conjecture \ref{conj} follows from Theorem \ref{wcbryan}. In contrast to the DT/PT correspondence, which is not known in non-Calabi-Yau cases, the BS/ PT correspondence can be established in some situations when $Y$ is not Calabi-Yau. The following is our main result:
\begin{thm}\label{thm2}
Conjecture \ref{conj} holds for $f:Y\to X$ the contraction of a curve $C\cong \mathbb{P}^1$ with normal bundle $N_{C/Y}\cong \OO_C(-1)^{\oplus 2}.$
\end{thm}

We next explain the main steps of the proof of Theorem \ref{thm2}:
\\

\textbf{Step 1.} In Section \ref{deg}, we define relative BS pairs and prove a degeneration formula for BS pairs similar to the degeneration formula for DT, PT pairs proved by Li--Wu \cite{lw}. 
\\

\textbf{Step 2.} In Section \ref{wc}, we use the degeneration formulas for BS and PT invariants for the family
\begin{equation}\label{fami2}
    \text{Bl}_{C\times 0}\left(Y\times\mathbb{A}_{\mathbb{C}}^1\right)\to \text{Bl}_{p\times 0}\left(X\times\mathbb{A}_{\mathbb{C}}^1\right)
    \end{equation}
    over $\mathbb{A}_{\mathbb{C}}^1$,
where $p\in X$ is the singular point. Let $g:\mathbb{P}\to \mathbb{P}'$ be the contraction of the curve $C$.
The fiber over zero is 
$$\text{Bl}_CY\cup_S\mathbb{P},$$
where we use the notations $\mathbb{P}:=\mathbb{P}_C\left(\OO(-1)^{\oplus 2}\oplus \OO\right)$ and $S:=\mathbb{P}_C\left(\OO(-1)^{\oplus 2}\right)$. 
When using the degeneration formulas for PT and BS invariants for the family \eqref{fami2}, the insertions will be in the $\text{Bl}_CY$ part. 
Theorem \ref{thm2} follows from a correspondence between generating series of relative invariants $\text{BS}^{g/S}(q)$ and $\text{PT}^{\mathbb{P}/S}(q)$
with no insertions. 
\\

\textbf{Step 3.} Also in Section \ref{wc}, we use the virtual localization theorem \cite{gp}, \cite{kr} for the torus $T\cong \left(\mathbb{C}^*\right)^2\subset \left(\mathbb{C}^*\right)^3$ acting on $\mathbb{P}$ and preserving the natural Calabi-Yau form on $\mathbb{P}\setminus S$. Let $0, \infty \in C$ be the $T$-invariant points.
The $T$-fixed BS or PT complexes $\left[\OO_Y\xrightarrow{s} F\right]$ restricted to $Y\setminus C$ will be ideal sheaves of a $T$-fixed curve which intersects $S$ transversely, so they will have certain partition profiles $\pi=(\pi_i)_{i=1}^4$ along four legs $(L_i)_{i=1}^4$ of $\mathbb{P}$ from $0$ and $\infty$.

Consider the moduli spaces of virtual dimension zero
$\text{BS}_n(\pi,m)$ and $\text{PT}_n(\pi,m)$
of $T$-fixed BS and PT complexes $\left[\OO_{\mathbb{P}}\xrightarrow{s} F\right]$ with fixed partition profile $\pi$, with 
$\text{ch}(F)=\left([\pi]+m[C], n\right)\in N_{\leq 1}(\mathbb{P})$, with $F$ intersecting $S$ transversely. Denote their generating series by $\text{BS}_{\pi}(q)$ and 
$\text{PT}_{\pi}(q)$. 
By the localization formula, the correspondence
between generating series of relative invariants $\text{BS}^{g/S}(q)$ and $\text{PT}^{\mathbb{P}/S}(q)$
follows from a correspondence between
$\text{BS}_{\pi}(q)$ and 
$\text{PT}_{\pi}(q)$. 
\\

\textbf{Step 4.} 
In Section \ref{hall}, we prove the following wall-crossing formula 
between the generating series of $T$-fixed BS and PT invariants with fixed partitions $\pi=(\pi_i)_{i=1}^4$ along the legs $L=(L_i)_{i=1}^4$:
\[\text{BS}_{\pi}(q)=\frac{\text{PT}_{\pi}(q)}{\text{PT}_0(q)}.\]
The proof follows very closely the argument in \cite{bs} using identities in the motivic Hall algebra of $W$, where $W$ is a toric Calabi-Yau $3$-fold which contains a curve $C\cong \mathbb{P}^1$ with normal bundle $N\cong\mathcal{O}_C(-1)^{\oplus 2}$ such that the four legs from the $\left(\mathbb{C}^*\right)^3$-fixed points $0, \infty\in C$ different from $C$ are proper. 
\smallskip

We remark that Bryan-Steinberg pairs are also known to be related to invariants defined via quasimaps. Let $f: S\to S':=\mathbb{C}^n/\Gamma$ be a minimal resolution of a type A singularity. In \cite{hl}, Henry Liu related Bryan-Steinberg pairs on $f: S\times \mathbb{C}\to S'\times\mathbb{C}$ and quasimaps to the Hilbert scheme of points in $S$. 

\subsection{Further directions.}
We hope the strategy for establishing the BS/ PT correspondence for non-Calabi-Yau situations can be used for larger classes of maps $f$, for example for $f$ a contraction of a curve $C$ with normal bundle of determinant $\omega_C$, or for $f$ with exceptional divisor $D$ a rational surface of normal bundle $\omega_D$. 

We also hope that the strategy exposed above can be used to establish non-Calabi-Yau correspondences between other theories, or to establish properties of PT invariants, such as those in \cite{BuMo}, beyond the Calabi-Yau case. 

\subsection{Acknowledgments}
I thank Davesh Maulik for many discussions about the present project and for suggesting the strategy to prove Theorem \ref{thm2}. I thank Chiu-Chu Melissa Liu for useful comments and suggestions for further directions. I thank Georg Oberdieck for pointing an issue with the proof of Proposition 3.7.

\section{Preliminaries}

\subsection{Notations and conventions}
All the stacks considered are defined over the complex numbers $\mathbb{C}$. For $X$ an algebraic space and $i\in\mathbb{Z}$, denote by $\text{Coh}\,(X)$ the abelian category of coherent sheaves on $X$, by $\text{Coh}_{\leq i}(X)$ its subcategory of sheaves with support of dimension $\leq i$ and define similarly $\text{Coh}_{i}(X)$ and $\text{Coh}_{\geq i}(X)$. Let $D^b(X)$ be the derived category of bounded complexes of coherent sheaves on $X$. For complexes $A, B\in D^b(X)$, denote by $\text{ext}^i_X(A, B)=\text{dim}\,\text{Ext}^i_X(A,B)$. 

For $f:Y\to X$, let $\mathbb{L}_{f}$ be the cotangent complex of $f$. For $\mathcal{X}$ an algebraic stack, let $\mathbb{L}_{\mathcal{X}}=\mathbb{L}_f$ for $f:\mathcal{X}\to\text{Spec}\,\mathbb{C}$. 

Let $Y$ be a smooth complex threefold. Denote by $\omega_Y$ its canonical line bundle. 
We say that $Y$ is \textit{Calabi-Yau} if $\omega_Y\cong \mathcal{O}_Y$ and $H^1(Y, \mathcal{O}_Y)=0$.



When considering generating series of BS or PT invariants with no insertions, we drop writing $\emptyset$ in the generating series, for example $\text{PT}^{Y, \text{exc}}(q)=\text{PT}^{Y, \text{exc}}(q; \emptyset, \emptyset)$.

\subsection{Perfect obstruction theories}
Let $X$ be a proper algebraic space. A two-term complex $$E^{\cdot}=\left[E^{-1}\to E^0\right]$$ of vector bundles on $X$ is called a \textit{perfect obstruction theory} for $X$ if there exists a morphism $$E^{\cdot}\to \mathbb{L}_X$$ in the derived category $D^b(X)$ which induces an isomorphism on $h^0$ and a surjection on $h^{-1}$. Let $d=\text{rk}\,E^0-\text{rk}\,E^{-1}$ be the rank of $E^{\cdot}$. A perfect obstruction theory induces a \textit{virtual fundamental class} $[X]^{\text{vir}}\in A_d(X), H_{2d}(X,\mathbb{Z})$, see \cite{bf}.
\\

\subsection{Symmetric perfect obstruction theories}\label{behrend}
Let $X$ be a proper algebraic space with a \textit{symmetric perfect obstruction theory} $$E^{\cdot}\to\mathbb{L}_X,$$ which means that $E^{\cdot}\to\mathbb{L}_X$ is a perfect obstruction theory and that there exists a non-degenerate symmetric bilinear form $\theta:E^{\cdot\dual}[1]\to E^{\cdot}$. By \cite[Subsection 1.1]{bf2}, we have that $\theta^{\vee}[1]=\theta$.
In this case, the virtual dimension $d$ is zero, so 
$$[X]^{\text{vir}}\in H_{0}(X,\mathbb{Z})\cong\mathbb{Z}.$$
Let $\nu:X\to\mathbb{Z}$ be the associated Behrend function \cite{be}. In loc. cit., Behrend proved that
$$[X]^{\text{vir}}=\sum_{n\in\mathbb{Z}} n\chi(\nu^{-1}(n)).$$ 
Let $M$ be a smooth algebraic space and $f:M\to\mathbb{C}$ a regular function with zero the only critical value. Let $X$ be the critical locus of $f$, and let $P\in X$. Then
$$\nu_X(P)=(-1)^{\text{dim}\,M}(1-\chi(M_P)),$$ where $M_P$ is the Milnor fiber of $f$ at $P$.
\\

\subsection{Localization}\label{localization}
Let $X$ be a proper algebraic space with an action of a torus $T$. Assume it has a $T$-equivariant perfect obstruction theory $E^{\cdot}$, meaning that $E^\cdot\in D^b_T(X)$ and the morphism $E^\cdot\to\mathbb{L}_X$ is $T$-equivariant.
Let $Z\subset X$ be a $T$-fixed algebraic space. Every vector bundle $V$ on $X$ splits over $Z$ as $V|_Z\cong V^f|_Z\oplus V^m|_Z,$ where $V^f|_Z\subset V|_Z$ is the sub-bundle where $T$ acts trivially and $V^m|_Z\subset V|_Z$ is the sub-bundle where $T$ acts with nonzero weight.

The $T$-fixed loci $Z\subset X$ admit perfect obstruction theories
$$E_Z:=\left[E^{-1,f}\big|_Z\to E^{0,f}\big|_Z\right].$$ The virtual normal bundle of the inclusion $Z\subset X$ is defined by $$N^{\text{vir}}:=\Big[E^{-1,m}\big|_Z\to E^{0,m}\big|_Z\Big].$$ 

Let $T$ act on $X$ with fixed connected components $X_k$ for $1\leq k\leq n$. The localization formula of Graber--Pandharipande \cite{gp}, Kresch \cite{kr} says that
$$[X]^{\text{vir}}=\sum^n_{k=1} \frac{[X_k]^{\text{vir}}}{e(N_k^{\text{vir}})}\in H_{\cdot}^T(X, \mathbb{Q})\otimes \text{Frac}\,H_T^{\cdot}(\text{pt})\cong H_{\cdot}\left(X^T, \mathbb{Q}\right)\otimes \text{Frac}\,H_T^{\cdot}(\text{pt}),$$ where $\text{Frac}\,H_T^{\cdot}(\text{pt})$ denotes the fraction field of the polynomial ring $H_T^{\cdot}(\text{pt})$.

\subsection{Laurent rings}\label{Laurent}
Let $Y$ be a smooth variety. Then $N_{\leq 1}(Y)$ is a finitely generated abelian group. Consider $\Delta\subset N_{\leq 1}(Y)$ the monoid of classes
$(\beta, n)$ where $\beta$ is effective or $\beta=0$ and $n\geq 0$.
Let $\mathbb{C}[\Delta]$ be the algebra with underlying $\mathbb{C}$-vector space generated by elements $q^{\beta+n}$ and with multiplication 
$$q^{\beta+n}q^{\gamma+m}=q^{(\beta+\gamma)+(n+m)}.$$
The algebra $\mathbb{C}[\Delta]$ is $\Delta$-graded. 
\\

\begin{defn}
For a $\Delta$-graded algebra $A$, define the Laurent completion $A_\Phi$ as follows: as a $\mathbb{C}$-vector space, it has elements
infinite series 
$$x=\sum_{(\beta,n)\in \Delta} x_{\beta, n}$$ such that for every $\beta\in N_1(Y)$, the set 
$\{n|\,x_{\beta,n}\neq 0\}$ is bounded from below. We call such infinite series \textit{Laurent}. 
For $(\beta, n)\in N_{\leq 1}(Y)$ and for $x$ a Laurent series as above, let $$\pi_{\beta,n}(x):=x_{\beta,n}.$$
The multiplication on $A_{\Phi}$ is defined as follows: for $x$ and $y$ Laurent series, $xy$ is the Laurent series such that for every $(\beta,n)\in N_{\leq 1}(Y)$, we have that
$$\pi_{\beta,n}(xy)=\sum_{(\beta,n)=(\gamma,m)+(\delta,p)}\pi_{\gamma,m}(x)\pi_{\delta,p}(y).$$ 
\end{defn}

Such a Laurent algebra comes with a natural topology by imposing that a sequence $(x_n)_{n\geq 0}$ of elements in $A_{\Phi}$ converges
if for every $(\beta, n)\in N_{\leq 1}(Y)$, there exists $K$ such that for every $i,j\geq K$ and $n\geq m$, we have that
$$\pi_{\beta,m}(x_i)=\pi_{\beta,m}(x_j).$$ The algebra $A_{\Phi}$ is a topological algebra.

\subsection{Torsion pairs}
Let $\textbf{C}$ be an abelian category with subcategories $\textbf{T}, \textbf{F}\subset \textbf{C}$. The pair $(\textbf{T}, \textbf{F})$ is called \textit{a torsion pair} if: 
\begin{itemize}
    \item For $T\in\textbf{T}$ and $F\in\textbf{F}$, we have that $\text{Hom}(T,F)=0$,
    \item For every $C\in \textbf{C}$, there exist (unique) $T\in\textbf{T}$ and $F\in\textbf{F}$ such that $$0\to T\to C\to F\to 0.$$
\end{itemize}

\subsection{Pandharipande-Thomas pairs}\label{pt}

Let $Y$ be a smooth threefold. We do not assume that $Y$ is projective.
Consider the abelian category
$\textbf{C}:=\text{Coh}_{\leq 1}(Y)$ 
of coherent sheaves with support of dimension at most $1$. Let $\textbf{T}:=\text{Coh}_{\leq 0}(Y)$ be the category of sheaves with support of dimension at most zero, and let 
$$\textbf{F}:=\{F\in \textbf{C}|\,\text{Hom}\,(T, F)=0\text{ for any }T\in \textbf{T}\}$$ be the orthogonal of $\textbf{T}$ in $\textbf{C}$.
The subcategories $(\textbf{T},\textbf{F})$ form a torsion pair of $\textbf{C}$.
\smallbreak

Let $(\beta, n)\in N_{\leq 1}(Y)$. Denote by $\text{PT}_n(Y,\beta)$ the moduli space of pairs $$I=\left[\mathcal{O}_Y\xrightarrow{s} F\right]$$ such that $\text{ch}\,(F)=(\beta, n)\in N_{\leq 1}(Y)$, $F\in \textbf{F}$, and  $\text{coker}\,(s)\in \textbf{T}.$ 

Assume that $Y$ is projective.
Then $\text{PT}_n(Y,\beta)$ is a projective variety. Consider the universal stable pair $$\mathbb{I}=\left[\mathcal{O}_{Y\times \text{PT}_n(Y,\beta)}\to\mathbb{F}\right]\in D^b\left(Y\times \text{PT}_n(Y,\beta)\right).$$ 
Let $R\mathcal{H}om(\mathbb{I},\mathbb{I})_0$ be the kernel of the trace map
$$\text{Tr}: R\mathcal{H}om(\mathbb{I},\mathbb{I})\to \mathcal{O}_{Y\times\text{PT}_n(Y,\beta)}.$$
The Atiyah class $\text{At}\in \text{Ext}^1\left(\mathbb{I}, \mathbb{I}\otimes \mathbb{L}_{Y\times \text{PT}_n(Y,\beta)}\right)$ induces a perfect obstruction theory $$E^{\cdot}:=R\pi_{2*}\big(R\mathcal{H}om(\mathbb{I},\mathbb{I})_0\otimes\pi_1^*\omega_Y [2]\big)\to \mathbb{L}_{\text{PT}_n(Y,\beta)}$$ using the results of Huybrechts--Thomas \cite{ht}, and thus a virtual fundamental class of dimension $$\text{dim}_\C [\text{PT}_n(Y,\beta)]^{\text{vir}}=-\chi\left(\text{RHom}(I, I)_0\right)=\beta\cdot c_1(Y).$$

\subsection{Generating series of PT invariants.}
Assume that $Y$ is projective.
Next, we define PT invariants with insertions. Recall the universal stable pair \[\mathbb{I}\in D^b(Y\times \text{PT}_n(Y,\beta)).\]
Consider a cohomology class $\gamma\in H^l(Y,\mathbb{Z})$ and an integer $k\geq 0$.
Define $$\ch_{2+k}(\gamma): H_*\left(\text{PT}_n(Y,\beta),\mathbb{Q}\right)\to H_{*-2k+2-l}\left(\text{PT}_n(Y,\beta),\mathbb{Q}\right)$$ by the formula $$\ch_{2+k}(\gamma)(-)=\pi_{2*}\left(\ch_{2+k}(\mathbb{I}) \pi_1^*(\gamma)\cap \pi_2^*(-)\right).$$ 
The PT invariants with insertions $\gamma_1,\cdots,\gamma_r\in H^{\cdot}(Y,\mathbb{Z})$ and descendant levels $\kappa_1,\cdots,\kappa_r\geq 0$ are defined by: 
\[\langle \tau_{\kappa_1}(\gamma_1),\cdots, \tau_{\kappa_r}(\gamma_r)\rangle_{\beta,n}=\int_{[\text{PT}_n(Y,\beta)]^{\text{vir}}} \ch_{2+\kappa_1}(\gamma_1)\cdots \ch_{2+\kappa_r}(\gamma_r).\]
The generating series for PT invariants of class $\beta$ with the given insertions and descendant levels is defined by:
$$\text{PT}^Y_{\beta}(q; \gamma, \kappa):=\sum_{n\in\mathbb{Z}} \langle \tau_{\kappa_1}(\gamma_1),\cdots, \tau_{\kappa_r}(\gamma_r)\rangle_{\beta,n}\, q^n.$$
Observe that $\text{PT}^Y_{\beta}(q; \gamma, \kappa)$ is a Laurent series in $q$, and thus the generating series 
\begin{align*}
    \text{PT}^Y(q; \gamma, \kappa)&:=\sum_{\beta\in N_{1}(Y)} \text{PT}^Y_{\beta}(q; \gamma, \kappa)q^{\beta},\\
    \text{PT}^{Y, \text{exc}}(q; \gamma, \kappa)&:=\sum_{\beta\in N_{ 1}^{\text{exc}}(Y)} \text{PT}^Y_{\beta}(q; \gamma, \kappa)q^{\beta}
\end{align*} are both elements of the Laurent ring $\mathbb{C}[\Delta]_{\Phi}$, see \cite[Lemma 5.5]{b}.

\subsection{Relative PT invariants}\label{relPT}
We review the theory of relative stable pairs, see \cite[Subsection 3.7]{pt} for definitions and desired properties and \cite{lw}, especially \cite[Subsection 6.4]{lw}, for proofs.

Let $S\subset Y$ be a smooth divisor with normal bundle $N_S$.
For $k\geq 1$, consider the $k$-step degeneration of $Y$:
$$Y[k]=Y\cup_S \mathbb{P}_S(N_S\oplus\OO)\cup_S \cdots\cup_S \mathbb{P}_S(N_S\oplus\OO),$$ where the union has $k$ copies of $\mathbb{P}_S(N_S\oplus\OO)$.
There are natural projection maps $\pi: Y[k]\to Y$.
Further, $Y[k]$ had $\left(\mathbb{C}^*\right)^k$ automorphisms covering the identity on $Y[0]$.

There is a Deligne-Mumford stack $\text{PT}_n(Y/S,\beta)$ 
parametrizing pairs $$\left[\mathcal{O}_{Y[k]}\xrightarrow{s} F\right],$$ where $F$ is a sheaf on $Y[k]$ such that $\pi_*[F]=\beta\in N_1(Y)$, $\chi(F)=n$, and:

\begin{itemize}
\item $F$ is pure and has a finite locally free resolution,

\item $F$ intersects the singular loci of $Y[k]$ and the relative divisor $S_{\infty}\subset Y[k]$ transversally,

\item $\text{coker}(s)$ has a zero dimensional support and is supported away from the singular loci of $Y[k]$,

\item the pair has only finitely many automorphisms covering the automorphisms of $Y[k]/Y$. 
\end{itemize}
 
 Assume that $Y$ is projective. The space $\text{PT}_n(Y/S,\beta)$ admits a perfect obstruction theory of the same dimension as that of $\text{PT}_n(Y,\beta)$.
 The complexes $$\left[\mathcal{O}_Y\xrightarrow{s} F\right]$$ where $F$ intersects $S$ transversely form an open set of $\text{PT}_n(Y/S,\beta)$, and the restriction of the perfect obstruction theory for this locus is the same as the perfect obstruction theory of $\text{PT}_n(Y,\beta)$. 
 Further, sending a sheaf $F$ to its restriction to $S$ defines a morphism $$\varepsilon: \text{PT}_n(Y/S,\beta)\to \text{Hilb}\,(S,d),$$ where $d=\beta\cdot S\in H_0(Y,\mathbb{Z})\cong\mathbb{Z}$.
 
Fix a basis $\beta_1,\cdots,\beta_m$ of $H^{\cdot}(S,\mathbb{Q})$. A \textit{cohomologically weighted partition} $\eta$ with respect to $(\beta_i)_{i=1}^m$ is a set of pairs
\begin{equation}\label{def:nu}
\{(\eta_1,\beta_{l_1}),\cdots, (\eta_s,\beta_{l_s})\},
\end{equation}
where $\eta_i$ are non-negative integers and $\sum_{i=1}^s \eta_i$ is an unordered partition of $\beta\cdot S=d$.
For a cohomologically weighted partition as above, let:
\begin{itemize}
    \item $l(\eta)=s$, 
    \item $|\eta|=\sum_{i=1}^s \eta_i=d$,
    \item $\text{Aut}\,(\eta)$ be the group of permutation symmetries of $\eta$, and
    \item $\xi(\eta)=\prod_{i=1}^s \eta_i|\text{Aut}(\eta)|$.
\end{itemize}
Assume that the basis $\beta_1,\ldots, \beta_m$ of is self dual with respect to the
Poincaré pairing. For $\beta\in H^\cdot(S, \mathbb{Q})$, denote by $\beta^{\vee}$ its Poincaré dual.
For a partition $\eta$ as in \eqref{def:nu}, let the dual partition $\eta^{\vee}$ be
the cohomologically weighted partition $\{(\eta_1, \beta_{l_1}^{\vee}),\ldots, (\eta_s, \beta_{l_s}^{\vee})\}$, see \cite[Subsection 3.2.2]{mnop2}.

Let $\{C_{\eta}\}_{|\eta|=d}$ be the basis of $H^{\cdot}(\text{Hilb}(S,d),\mathbb{Q})$ introduced in \cite[Section 3.2.2]{mnop2} and constructed from the basis $\beta_1,\ldots, \beta_m$ of $H^\cdot(S, \mathbb{Q})$.
For insertions and descendant levels as above,
the relative PT invariants are defined as follows
\begin{equation}\label{pt33}
\langle \tau_{\kappa_1}(\gamma_1),\cdots, \tau_{\kappa_r}(\gamma_r)|\,\eta\rangle^{\text{PT}}_{\beta,n}:=\int_{[\text{PT}_n(Y/S,\beta)]^{\text{vir}}} \ch_{2+\kappa_1}(\gamma_1)\cdots \ch_{2+\kappa_r}(\gamma_r)\cap \varepsilon^*(C_{\eta}).
\end{equation}
The generating series for the relative PT invariants for fixed class $\beta$, insertions $\gamma_1,\cdots, \gamma_r$, and descendants $\kappa_1,\cdots, \kappa_r$ is defined by the Laurent series in $q$:
\begin{equation}\label{PTrel}
    \text{PT}^{Y/S}_{\beta,\eta}(q; \gamma, \kappa):=\sum_{n\in\mathbb{Z}} \langle \tau_{\kappa_1}(\gamma_1),\cdots, \tau_{\kappa_r}(\gamma_r)|\,\eta\rangle^{\text{PT}}_{\beta,n}\, q^n.
    \end{equation}

\subsection{The degeneration formula for PT invariants}
Let $\mathcal{Y}\to C$ be a smooth fourfold fibered over a smooth curve. 
Let $Y$ be a nonsingular fiber, and assume that there is a singular fiber over $0\in C$ of the form $$\mathcal{Y}_0=Y_1\cup_S Y_2,$$ where $Y_1$ and $Y_2$ are two smooth threefolds intersecting in a smooth surface $S$. Consider the natural inclusions $i:Y\to\mathcal{Y}$, $i_0:\mathcal{Y}_0\to\mathcal{Y}$, $i_1:Y_1\to\mathcal{Y}_0$, $i_2:Y_2\to\mathcal{Y}_0$. Further, consider the morphisms
$$H_2(Y)\xrightarrow{i_*} H_2(\mathcal{Y})\xleftarrow{i_{0*}} H_2(\mathcal{Y}_0)\xleftarrow{i_{1*}+i_{2*}} H_2(Y_1)\oplus H_2(Y_2).$$
The map $i_{0*}$ is an isomorphism. 
Let $\beta\in H_2(Y)$ be a curve class.
We say that $\beta$ splits as $\beta=\beta_1+\beta_2$ for classes $\beta_1\in H_2(Y_1)$, $\beta_2\in H_2(Y_2)$ if 
$$i_*(\beta)=i_{0*}\left(i_{1*}(\beta_1)+i_{2*}(\beta_2)\right).$$
Consider the insertions $\gamma_1,\cdots,\gamma_r\in H^{\cdot}(\mathcal{Y},\mathbb{Z})$ and descendants levels $\kappa_1,\cdots, \kappa_r\geq 0$. We abuse notation and denote the restrictions of the insertions to $Y$, $Y_1$, and $Y_2$ also by $\gamma_1,\cdots, \gamma_r$.
The degeneration formula \cite[Theorem 6.9]{lw}, \cite{mpt} relates the PT invariants of $Y$ to the relative PT invariants of $Y_1/S$ and $Y_2/S$:
$$\text{PT}^Y_{\beta}(q; \gamma, \kappa)=\sum \text{PT}^{Y_1/S}_{\beta_1,\eta}(q; \gamma, \kappa)\,\text{PT}^{Y_2/S}_{\beta_2,\eta^{\vee}}(q; \gamma, \kappa)
\frac{(-1)^{|\eta|-l(\eta)}\xi(\eta)}{q^{|\eta|}},$$
where the sum on the right hand side is taken over all splittings $\beta=\beta_1+\beta_2$ of the curve class and over all cohomologically weighted partitions $\eta$. 

\subsection{Stability conditions}\label{stabi}
Let $Y$ be a smooth projective variety. \textit{A stability condition} on $\text{Coh}_{\leq 1}(Y)$ consists of a slope function \[\mu: \text{Coh}_{\leq 1}(Y)\to S\] where $(S,\leq )$ is a totally ordered set, such that
\begin{enumerate}
    \item for any exact sequence $0\to A\to B\to C\to 0$ of objects in $\text{Coh}_{\leq 1}(Y)$, we have that either
    \begin{align*}
        &\mu(A)<\mu(B)<\mu(C)\text{ or }\\
        &\mu(A)=\mu(B)=\mu(C)\text{ or }\\
        &\mu(A)>\mu(B)>\mu(C).
    \end{align*}
 \item  any sheaf $F\in \text{Coh}_{\leq 1}(Y)$ has a Harder-Narasimhan filtration $$0\subset F_1\subset \cdots\subset F_{n-1}\subset F_n=F$$ such that the factors $\text{gr}^iF=F_i/F_{i-1}$ are semistable and the slopes of the factors satisfy $\mu(\text{gr}^1F)>\cdots>\mu(\text{gr}^nF)$,
\end{enumerate}
 see \cite[Section 3]{bs}, \cite[Section 4]{j4}.
A sheaf $F\in \text{Coh}_{\leq 1}(Y)$ is called \textit{(semi)stable} if for any proper subsheaf $0\neq E\subset F$, we have $\mu(E)(\leq)<\mu(F).$

For $s\in S$ define:
\begin{align*}
    \textbf{T}_s:=\{T\in \text{Coh}_{\leq 1}(Y)|\, T\twoheadrightarrow Q\neq 0,\text{ then }\mu(Q)\geq s\}\\
    \textbf{F}_s:=\{F\in \text{Coh}_{\leq 1}(Y)|\,0\neq S\hookrightarrow F,\text{ then }\mu(S)<s\}.
\end{align*}
The categories $\left(\textbf{T}_s, \textbf{F}_s\right)$ form a torsion pair of $\text{Coh}_{\leq 1}(Y)$. 
The categories $\textbf{T}_s$ and $\textbf{F}_s$ can be also described as follows. For a set $A\subset S$, let $\text{SS}(A)\subset \text{Coh}_{\leq 1}(Y)$ be the subcategory generated by semistable sheaves of slope in $A$. 
For $s\in S$, we have that 
\begin{align*}
    \textbf{T}_s:=\text{SS}(\geq s),\\
    \textbf{F}_s:=\text{SS}(<s).
\end{align*}

\subsection{Moduli stacks and torsion pairs}\label{moduli}
Let $Y$ be a smooth projective variety. Lieblich \cite{li} constructed an Artin stack $\mathfrak{LM}_Y$, locally of finite type, parametrizing gluable complexes $I\in D^b(Y)$, that is, complexes $I$ such that $\text{Ext}^{\leq -1}_Y(I,I)=0$.

Consider the abelian category defined by Toda \cite{t1}:
$$\textbf{A}=\langle \OO_Y[1], \text{Coh}_{\leq 1}(Y)\rangle_{\text{exc}}\subset D^b(Y).$$ 
For a torsion pair $(\textbf{T}, \textbf{F})$ of $\text{Coh}_{\leq 1}(Y)$, we call an object $I\in \textbf{A}$ \textit{a $(\textbf{T}, \textbf{F})$-pair} if it is of the form
$$I\cong \left[\OO_Y\xrightarrow{s} F\right],$$ where $F\in\textbf{F}$ and $\text{coker}(s)\in \textbf{T}$. Further, a torsion pair $(\textbf{T}, \textbf{F})$ of $\text{Coh}_{\leq 1}(Y)$ is called \textit{open} if the categories $\textbf{T}, \textbf{F}\subset \text{Coh}_{\leq 1}(Y)$ are open.

The following result is proved in \cite[Lemma 4.6]{bcr}; it is originally stated for some particular cases of Calabi-Yau orbifolds $\mathcal{X}$, but the proof in loc. cit. works also for smooth projective varieties $Y$.

\begin{lemma}\label{openimp}
Let $(\textbf{T}, \textbf{F})$ be an open torsion pair for $\text{Coh}_{\leq 1}(Y)$ and assume that $\text{Coh}_{ 0}(Y)\subset \textbf{T}$. The substack of $\mathfrak{LM}$ parametrizing $(\textbf{T}, \textbf{F})$-pairs is open.
\end{lemma}

\begin{proof}
The proof of Lemmas 4.1 and 4.2 in \cite{bcr} translate directly in our case.  
Consider the dual $\mathbb{D}(-):=R\mathcal{H}om(-, \omega_Y)[2]$ of $D^b(Y)$. 
Then $\mathbb{D}\left(\text{Coh}_1(Y)\right)=\text{Coh}_1(Y)$,
$\mathbb{D}\left(\text{Coh}_0(Y)\right)=\text{Coh}_0(Y)[-1]$, 
and $\mathbb{D}(\mathcal{O}_Y)=\omega_Y[2]$. The duality functor $\mathbb{D}$ induces an automorphism of the stack $\mathfrak{LM}$. The proofs of Lemmas 4.4 and 4.5 and Proposition 4.6 in loc. cit can be changed for $Y$ a smooth projective threefold. 
\end{proof}


\subsection{The motivic Hall algebra}\label{moHA}
We recall the construction of Joyce's motivic Hall algebras from \cite{j3}, see also \cite{b}. Let $\mathfrak{M}$ be an Artin stack locally of finite type over $\mathbb{C}$ with affine stabilizers.
\smallbreak

\begin{defn}
 \textit{The relative Grothendieck group over $\mathfrak{M}$}, denoted by $K(\text{St}/\mathfrak{M})$, is the $\mathbb{C}$-vector space of equivalence classes of symbols $\left[T\xrightarrow{f_T} \mathfrak{M}\right]$, where $T$ is a locally finite-type stack with affine stabilizers and $f_T$ a morphism of stacks, modulo the relations:
 
\begin{itemize}
    \item for $T$ and $S$ Artin stacks over $\mathfrak{M}$ with affine stabilizers, we have $$\left[T\cup S\xrightarrow{f_T\cup f_s}\mathfrak{M}\right]=\left[T\xrightarrow{f_T}\mathfrak{M}\right]+\left[S\xrightarrow{f_S}\mathfrak{M}\right],$$
    \item for $f:T\to S$ a map of stacks over $\mathfrak{M}$ with affine stabilizers such that $f:T(\mathbb{C})\to S(\mathbb{C})$ is an equivalence of groupoids, we have $$\left[T\xrightarrow{f_T}\mathfrak{M}\right]=\left[S\xrightarrow{f_S}\mathfrak{M}\right],$$
    \item and for $h_T: T\to U$ and $h_S: S\to U$ Zariski fibrations over $\mathfrak{M}$ of the same dimension, we have $$\left[T\xrightarrow{f_Uh_T}\mathfrak{M}\right]=\left[S\xrightarrow{f_Uh_S}\mathfrak{M}\right].$$
\end{itemize}
\end{defn}

Let $\textbf{A}$ be an abelian category and assume that the moduli stack $\mathfrak{M}$ of objects in $\textbf{A}$ is as above.
Let $\mathfrak{M}^{(2)}$ the stack of short exact sequences with elements in $\textbf{A}$. There are natural maps
\[ \mathfrak{M}\times\mathfrak{M}\xleftarrow{q}\mathfrak{M}^{(2)}\xrightarrow{p} \mathfrak{M},\]
where $p(0\to A\to B\to C\to 0)=(A,C)$ and $q(0\to A\to B\to C\to 0)=B$.

\begin{defn}
The \textit{motivic Hall algebra of $\textbf{A}$} is the $\mathbb{C}$-vector space $K(\text{St}/\mathfrak{M})$ with the product
$$[S\to\mathfrak{M}]*[T\to\mathfrak{M}]:=[E(S,T)\to \mathfrak{M}],$$ where $E(S,T)\to \mathfrak{M}$ is defined by the diagram
\begin{equation*}
\begin{tikzcd}
E(S,T) \arrow[d] \arrow[r] & \mathfrak{M}^{(2)} \arrow[d] \arrow[r] & \mathfrak{M} \\
S\times T \arrow[r] & \mathfrak{M}\times \mathfrak{M},
\end{tikzcd}
\end{equation*}
where the square is cartesian. Further,
$K(\text{St}/\mathfrak{M})$ is an algebra over $K(\text{St}/k)$ by letting $$[S\to k]\times [T\to\mathfrak{M}]:=[S\times T\to \mathfrak{M}].$$ 
\end{defn}

\begin{defn} The \textit{Grothendieck ring of varieties} $K(\text{Var}/\mathbb{C})$ is the $\mathbb{C}$-vector space of equivalence classes of symbols $[X\to \text{Spec}\,\mathbb{C}]$ for $X$ a complex variety with relations
$$[X\to\text{Spec}\,\mathbb{C}]=[U\to\text{Spec}\,\mathbb{C}]+[Z\to\text{Spec}\,\mathbb{C}],$$ where $X$ is a variety over $\mathbb{C}$, $Z\subset X$ is closed, and $U:=X\setminus Z$ is its complement.
\end{defn}

One can similarly define the Grothendieck ring of algebraic spaces $K(\text{Sp}/\mathbb{C})$. There is a natural isomorphism 
\begin{equation}\label{groth}
    K(\text{Var}/\mathbb{C})\cong K(\text{Sp}/\mathbb{C}),
\end{equation} see \cite[Lemma 2.12]{b2}.

\section{The moduli space of BS pairs}\label{fund}

In this Section, we prove Theorem \ref{thm1}. In Subsection \ref{defs}, we introduce the torsion pair $(\textbf{T}, \textbf{F})$ and define Bryan--Steinberg pairs. 
In Subsection \ref{msbs}, we prove that the moduli of BS pairs under the hypothesis of Theorem \ref{thm1} is a proper algebraic space. In Subsection \ref{vfc}, we use the results of Huybrechts--Thomas \cite{ht} to construct a virtual fundamental class for the moduli of BS pairs and define BS invariants.

\subsection{Definition of Bryan-Steinberg pairs}
\label{defs}
Let $Y$ be a smooth threefold with a birational morphism $f:Y\to X$ with $Rf_*\mathcal{O}_Y=\mathcal{O}_X$.
Such a map $f:Y\to X$ determines a torsion pair $(\textbf{T},\textbf{F})$ for $\text{Coh}_{\leq 1}(Y)$ \cite{bs}, defined as follows: $\textbf{T}\subset \text{Coh}_{\leq 1}(Y)$ is the subcategory of sheaves $T$ such that 
\begin{equation}\label{RfCoh}
Rf_*T\in \text{Coh}_{\leq 0}(X)
\end{equation} 
and $\textbf{F}$ is its complement
$$\textbf{F}=\{F\in \text{Coh}_{\leq 1}(Y)|\, \text{Hom}\,(T,F)=0\text{ for any }T\in\textbf{T}\}.$$ 
The statement \eqref{RfCoh} means that $R^{>0}f_*T=0$ and the sheaf $f_*T=R^0f_*T\in \text{Coh}_{\leq 0}(X)$.
A complex $I=\left[\mathcal{O}_Y\xrightarrow{s} F\right]$ is called a \textit{Bryan-Steinberg (BS) pair} if $F\in\textbf{F}$ and $\text{coker}(s)\in\textbf{T}$. 
\begin{prop}
The pair $(\textbf{T},\textbf{F})$ is a torsion pair.
\end{prop}

\begin{proof}
This follows as \cite[Lemma 13]{bs}. By \cite[Lemma 10]{bs}, it suffices to show that $\textbf{T}$ is closed under extension and under taking quotients. $\textbf{T}$ is clearly closed under extensions. Next, let $T\in \textbf{T}$, let $T\twoheadrightarrow V$ with kernel $K$. Then 
\[0\to f_*K\to f_*T\to f_*V\to R^1f_*K\to R^1f_*T\to R^1f_*V\to 0.\]
The sheaf $K$ has support of dimension at most one. If $x\in X$ is in the support of $R^1f_*K$, then there exists a curve $C\subset f^{-1}(x)$ such that $C$ is in the support of $K$. If $R^1f_*K$ has support of dimension at least one, we then obtain that $K$ has support of dimension at least $2$, which is not possible.
Thus $R^1f_*K$ has support of dimension at most zero. This means that $f_*V$ is zero dimensional. Finally, we have 
$R^1f_*T=0$, so $R^1f_*V=0$, and thus $V\in \textbf{T}$.
\end{proof}
Assume from now on that $Y$ is projective.
The torsion pair $(\textbf{T}, \textbf{F})$ can be obtained as a torsion pair for a stability condition $\mu:\text{Coh}_{\leq 1}(Y)\to S$, see Subsection \ref{stabi}. Consider the set $S=(-\infty,\infty]\times(-\infty,\infty]$, ordered lexicographically. Fix $L$ an ample line bundle on $X$, and fix $H>f^*L$ an ample line bundle on $Y$. Consider the slope map
$$\mu:\text{Coh}_{\leq 1}(Y)\to S$$ defined by the formula
$$\mu(F):=\left(\frac{\chi(F)}{\text{supp}(F)\cdot f^*L}, \frac{\chi(F)}{\text{supp}(F)\cdot H}\right),$$ where $\text{supp}(F)\in N_1(Y)$ is the (curve) support of $F$. This slope defines a stability condition on $\text{Coh}_{\leq 1}(Y)$; this is proved as in \cite[Section 3, Lemma 39]{bs}. For $a=(\infty, 0)$, the pair $(\textbf{T}_a, \textbf{F}_a)$ is the torsion pair used in the definition of BS pairs, see the argument in \cite[Lemma 51]{bs}. For $b=(\infty, \infty)$, the pair $(\textbf{T}_b, \textbf{F}_b)$ is the torsion pair used in the definition of PT pairs. 

\subsection{The moduli space of BS pairs}\label{msbs}
Let $I$ be a BS pair as above.
Consider the distinguished triangle
\begin{equation}\label{tri1}
    I\to \mathcal{O}_Y\to F\rightarrow I[1].
\end{equation}
Then $h^0(I)=\mathcal{I}_C$ for a one dimensional subscheme $C\subset Y$. Let $h^1(I)=Q$.
The BS pair $I$ also fits in a distinguished triangle
\begin{equation}\label{tri2}
    \mathcal{I}_C\to I\to Q[-1]\rightarrow \mathcal{I}_C[1].
\end{equation}
Consider the functor $\Phi_{\text{BS}(Y)}:(\text{Schemes}/k)^{\text{op}}\to\text{Sets}$ parametrizing BS pairs 
$$\Phi_{\text{BS}(Y)}(B)=\left\{\left[\OO_{Y\times B}\xrightarrow{s}\mathcal{F}\right]\text{ s.t.}\left[\OO_{Y\times b}\xrightarrow{s} \mathcal{F}|_{Y\times b}\right] \text{ is a BS pair for every }b\in B\right\}/\text{eq.}$$ where two families $\mathcal{F}$ and $\mathcal{F}'$ are equivalent if there exists a line bundle $\mathcal{L}$ on $B$ such that $\mathcal{F}\cong \mathcal{F}'\otimes\pi_2^*\mathcal{L}$.


As explained in the Introduction, the existence of the proper algebraic space $\text{BS}^f_n(Y,\beta)$
follows once we show that the functor $\Phi_{\text{BS}(Y)}$ is bounded, open, separated, complete, and has trivial automorphisms.
We first check that the BS pairs are in the Lieblich stack $\mathfrak{LM}$:

\begin{prop}
Let $I=\left[\OO_Y\xrightarrow{s}F\right]$ be a BS pair. Then $\text{Ext}^{\leq -1}(I,I)=0$.
\end{prop}

\begin{proof}
Apply $\text{Hom}(I,-)$ to the triangle (\ref{tri1}) to get a long exact sequence
$$\cdots\to\text{Ext}^{i-1}(I, F)\to \text{Ext}^{i}(I,I)\to \text{Ext}^i(I,\OO_Y)\to\cdots.$$ It thus suffices to show that $\text{Ext}^{\leq -1}(I, \OO_Y)\cong\text{Ext}^{\leq -2}(I, F)=0$. For this, apply $\text{Hom}(-,\OO_Y)$ and $\text{Hom}(-,F)$ to the triangle (\ref{tri2}) to get
\begin{align*}
\cdots&\to\text{Ext}^{i+1}(Q, \OO_Y)\to \text{Ext}^i(I, \OO_Y)\to \text{Ext}^i(\mathcal{I}_C,\OO_Y)\to\cdots\\
\cdots&\to\text{Ext}^{i+1}(Q, F)\to \text{Ext}^i(I, F)\to \text{Ext}^i(\mathcal{I}_C,F)\to\cdots
\end{align*}
The groups $\text{Ext}^{\leq 0}(Q, \OO_Y)$, $\text{Ext}^{\leq -1}(\mathcal{I}_C,\OO_Y)$, $\text{Ext}^{\leq 0}(Q, F)$, $\text{Ext}^{\leq -1}(\mathcal{I}_C,F)$ vanish and thus the conclusion follows.
\end{proof}

We next check that BS pairs have trivial automorphisms:

\begin{prop}\label{aut}
Consider two BS pairs $I=[\OO_Y\to F], J=[\OO_Y\to G]\in D^b(Y).$ 
The natural map $\text{Hom}(J,I)\to \text{Hom}(\OO_Y,\OO_Y)\cong \C$ is injective. In particular, morphisms in $\text{Hom}(I,I)$ are given by scalar multiplications.
\end{prop}

\begin{proof}
Consider the long exact sequence
$$\cdots\to \text{Hom}(G, \mathcal{O}_Y)\to \text{Hom}(\mathcal{O}_Y, \mathcal{O}_Y)\to \text{Hom}(J, \mathcal{O}_Y)\to \text{Ext}^1(G, \mathcal{O}_Y)\to\cdots$$ obtained by applying $\text{Hom}(-,\OO_Y)$ to the triangle \eqref{tri1}.  The sheaf $G$ has support of codimension at least $2$ in $Y$,
so $\text{Hom}(G, \mathcal{O}_Y)=\text{Ext}^1(G, \mathcal{O}_Y)=0$.
There is thus an isomorphism $$\C\cong\text{Hom}(\mathcal{O}_Y, \mathcal{O}_Y)\cong \text{Hom}(J, \mathcal{O}_Y).$$
Further, consider the exact sequence
$$\cdots\to\text{Ext}^{-1}(J,F)\to \text{Hom}(J,I)\to \text{Hom}(J,\mathcal{O}_Y)\to\cdots$$ obtained by applying $\text{Hom}(J,-)$ to the triangle \eqref{tri1}.
It is enough to show that $\text{Ext}^{-1}(J,F)=0$. 
There is a long exact sequence
 $$\cdots\to\text{Hom}(K,F)\to \text{Ext}^{-1}(J,F)\to \text{Ext}^{-1}(\mathcal{I}_D,F)\cong 0$$
 obtained by applying $\text{Hom}(-,F)$ to the distinguished triangle \eqref{tri2} for the pair $J$ $$\mathcal{I}_D\to J\to K[-1]\rightarrow \mathcal{I}_D[1].$$
 The first term is zero because $K\in \textbf{T}$ and $F\in\textbf{F}$, so we obtain that $\text{Ext}^{-1}(J,F)=0$.
\end{proof}


\begin{prop}\label{open}
The locus of BS pairs inside $\mathfrak{LM}$ is open. In particular,
the functor $\Phi_{\text{BS}(Y)}$ is open.
\end{prop}

\begin{proof}
By Lemma \ref{openimp}, the statement follows once we show that 
the categories $\textbf{T}, \textbf{F}\subset \text{Coh}_{\leq 1}(Y)$ are open. This follows as openness of semistable sheaves \cite[Proposition 2.3.1]{hl2}. 
\end{proof}

Recall the definition of boundedness from \cite[Definition 41]{bs}. The following is proved in the same way as \cite[Lemma 46]{bs}:

\begin{prop}
The family of BS pairs $\left[\mathcal{O}_Y\rightarrow F\right]$ for $[F]=(\beta, n)\in N_{\leq 1}(Y)$ is bounded.
\end{prop}

Next, we discuss that the functor $\Phi_{\text{BS}(Y)}$ is separated and complete. The proof is similar to Langton's proof that the moduli of Gieseker semistable sheaves is separated and proper, see \cite[Appendix 2.B]{hl2}, \cite{l}.
Let $R$ be a DVR with fraction field $K$, residue field $k$, and uniformizer $\pi$. Let $\mathcal{Y}:=Y\times R$, then $\mathcal{Y}_K=Y\times K$.

\begin{prop}\label{sep}
Consider a BS family $I=\left[\OO_{\mathcal{Y}_ K}\rightarrow F\right]$ over $K$. There exists at most one $R$-flat BS family $\I=\left[\OO_{\mathcal{Y}}\rightarrow \F\right]$ such that $\I_K\cong I$.
\end{prop}

\begin{proof}
Consider two $R$-flat BS families
\begin{align*}
    \I^1&=\left[\mathcal{O}_{\mathcal{Y}}\to\mathcal{F}^1\right],\\
    \I^2&=\left[\mathcal{O}_{\mathcal{Y}}\to \mathcal{F}^2\right]
\end{align*} such that $\I^1_K\cong \I^2_K\cong I$. The $R$-module
$\text{Hom}_{\mathcal{Y}}\left(\I^1, \I^2\right)$ has $\text{Hom}_{\mathcal{Y}}\left(\I^1, \I^2\right)\otimes_R K\cong \text{Hom}_{\mathcal{Y}_K}\left(\I^1_K, \I^2_K\right)$.
A relative version over $R$ of Proposition \ref{aut} shows that
$$\text{Hom}_{\mathcal{Y}}\left(\I^1, \I^2\right)\hookrightarrow R.$$ The left hand side above $\text{Hom}_{\mathcal{Y}}\left(\I^1, \I^2\right)$ is an $R$-module.
The isomorphism $\I^1_K\cong \I^2_K$ is a section of $\text{Hom}_{\mathcal{Y}_K}\left(\I^1_K, \I^2_K\right)\cong \text{Hom}_{\mathcal{Y}}\left(\I^1, \I^2\right)_K$, so there exists a global $R$-section of $\text{Hom}_{\mathcal{Y}}\left(\I^1, \I^2\right)$, that is, there exists a morphism $\I^1\to \I^2$ which extends the isomorphism over $K$. In particular, this means that
\begin{equation}\label{iso1}
    \text{Hom}_{\mathcal{Y}}\left(\I^1, \I^2\right)\cong R.
    \end{equation}
Consider the morphisms \begin{align*}
    \alpha&: \I^1\to \I^2,\\
    \beta&: \I^2\to\I^1
\end{align*}
corresponding to $1\in R\cong \text{Hom}_{\mathcal{Y}}\left(\I^1, \I^2\right)$ and $1\in R\cong\text{Hom}_{\mathcal{Y}}\left(\I^1, \I^2\right)$, respectively. 
Their composite $\varphi=\beta\alpha: \I^1\to \I^1$ restricts to an isomorphism $\I^1_K\cong \I^1_K\cong I$. By the isomorphism \eqref{iso1} where both complexes are $\mathcal{I}^1$, there is a natural isomorphism
$\text{Hom}_{\mathcal{Y}}\left(\I^1,\I^1\right)\cong R.$
Consider the diagram
\begin{equation*}
\begin{tikzcd}
\text{Hom}_{\mathcal{Y}}\left(\I^1,\I^2\right) \times \text{Hom}_{\mathcal{Y}}\left(\I^2,\I^1\right)\arrow[d]\arrow[r] & \text{Hom}_{\mathcal{Y}}\left(\I^1,\I^1\right) \arrow[d] \\
R\times R \arrow[r] & R.
\end{tikzcd}
\end{equation*}
The map $\varphi$ corresponds to $1\in R$,
so the composition 
\[\varphi_k:\I^1_k\to \I^2_k\to \I^1_k\] is an isomorphism. Both $\I^1$ and $\I^2$ have the same Hilbert polynomial, so they are isomorphic.
\end{proof}

\begin{prop}\label{com}
Let $I=\left[\mathcal{O}_{\mathcal{Y}_K}\to F\right]$ be a BS pair. There exists a flat BS family $\I=\left[\OO_{\mathcal{Y}}\to\F\right]$ over $R$ such that $\I_K\cong I$.
\end{prop}

\begin{proof}

Let $\HH$ be a $R$-flat extension of $F$. Then the $R$-module $\text{Hom}_{\mathcal{Y}}\left(\OO_{\mathcal{Y}}, \HH\right)$ has a section over $K$, so it has a non-zero global section
$$\OO_{\mathcal{Y}}\xrightarrow{s} \HH.$$ All subsheaves of $\HH$ are flat over $R$.
\\

\textbf{Step 1.} We first show that there exists a subsheaf $\HH'\subset \HH$ such that $\HH'_k\in\textbf{F}$. Assume that this not the case.
In particular, $\HH_k$ is not in $\textbf{F}$. Let $P_0\in\textbf{F}$ and $Q_0\in\textbf{T}$ be such that $$0\to Q_0\to \mathcal{H}_k\to P_0\to 0.$$ The sheaf $\HH_k$ is not in $\textbf{F}$, so $Q_0\neq 0$. 
Let $\mathcal{H}^1\subset \HH$ be the kernel of the map $\HH\to\HH/\pi\cong\HH_k\to P_0$, so there is a short exact sequence
$$0\to \mathcal{H}^1\to\mathcal{H}\to P_0\to 0.$$
The restriction of the above sequence to the central fiber gives an exact sequence:
\begin{equation}\label{tor}
    0\to P_0\to \HH^1_k\to\HH_k\to P_0\to 0.
\end{equation} We thus have a short exact sequence
$$0\to P_0\to\HH^1_k\to Q_0\to 0.$$
By our assumption, the sheaf $\HH^1_k$ is not in $\textbf{F}$. Thus, there exists a short exact sequence with $Q_1\in\textbf{T}$ and $P_1\in\textbf{F}$:
$$0\to Q_1\to \mathcal{H}^1_{k}\to P_1\to 0.$$
Consider the pushforward of both these sequences to $X$:
\begin{align*}
    0&\to f_*P_0\to f_*\HH^1_k\to f_*Q_0,\\
    0&\to f_*Q_1\to f_*\mathcal{H}^1_{k}\to f_*P_1\to 0.
\end{align*} 
Let $L:=\text{ker}\left(f_*Q_1\to f_*\mathcal{H}^1_{k}\to f_*Q_0\right).$ By diagram chasing, we see that $L\subset f_*P_0$, so we obtain a map $f^*L\to P_0$. The sheaf $f^*L$ is in $\textbf{T}$ because $Rf_*f^*L=L$, which means that $L=0$ and thus $f_*Q_1\subset f_*Q_0$. 
We repeat the procedure above to define $\HH^n\subset \HH^{n-1}$ for every $n\geq 1$, and the above arguments implies that we have a descending sequence of torsion sheaves on $X$:
$$f_*Q_n\subset f_*Q_{n-1}.$$ 
There exists $n_0\geq 0$ such that the above inclusions are equalities for $n\geq n_0$.
Let $n\geq n_0$, and define $$M:=\text{ker}\left(Q_{n+1}\to \HH^{n+1}_k\to Q_{n}\right).$$ 
We have $f_*Q_{n+1}\cong f_*Q_n$ and $R^1f_*Q_{n+1}= R^1f_*Q_n=0$. This implies that $$Rf_*M=0,$$ and thus that $M\in\textbf{T}$.
A diagram chasing implies that $M\subset P_{n}$, which is possible only if $M=0$. Thus $$Q_n\subset Q_{n-1}$$ for $n\geq n_0$. There exists $n_1\geq 0$ such that the above inclusion is an equality for $n\geq n_1$.
Without loss of generality, we can assume that $n_1=1$. This implies that the short exact sequences $$0\to Q_n\to \HH^{n+1}_k\to P_n\to 0$$ split for $n\geq 1$. This means that $\HH^{n+1}_k\cong Q_n\oplus P_n$ and thus that $P_{n}\cong P_{n-1}$ for $n\geq 0$. 

By induction on $n\geq 1$, the sheaf $\mathcal{H}/\mathcal{H}^n$ is flat over $R/\pi^n$ and there is a surjection
$$\mathcal{H}/\pi^n\twoheadrightarrow \mathcal{H}/\mathcal{H}^n.$$
Let $p$ be the Hilbert polynomial of $P_0$. The proper map
$$\tau: \text{Quot}(\mathcal{H},p)\to\text{Spec}(R)$$ contains $\text{Spec}\left(R/\pi^n\right)$ in its image for every $n\geq 1$ because the maps $$\text{Quot}\left(\mathcal{H}/\pi^n,p\right)\to \text{Spec}\left(R/\pi^n\right)$$ are all surjective.
This means that the map $\tau$ is surjective, so there exists a quotient $\HH\twoheadrightarrow P$, where $P$ is flat over $R$ and has Hilbert polynomial $p$. 
Let $$Q:=\text{ker}\,(\HH\to P)$$ be the kernel. Then
$Q_0\in \textbf{T}$, so $\mu(Q)=\mu(Q_0)\geq a$. 
Further, $\mu(Q)=\mu(Q_K)<a$ because $Q_K\in\textbf{F}$, so we obtain a contradiction.
This means that our assumption in the beginning of Step $1$ was false. Replace $\mathcal{H}$ with $\HH'$. Thus there exists a flat extension $\HH$ of $F_K$ over $R$ such that $\HH_k\in\textbf{F}$. 
\\

\textbf{Step 2.} The section $\OO_{\mathcal{Y}_K}\to F$ extends to a section $\OO_{\mathcal{Y}}\xrightarrow{s} \HH.$
Let $\mathcal{K}:=\text{coker}\,(s)$.
We next claim that there exists a subsheaf $\HH'\subset\HH$ such that $s$ factors through $\OO_{\mathcal{Y}}\xrightarrow{s'} \HH'$ and $\K':=\text{coker}\,(s')$ fits in a short exact sequence $0\to \K'_f\to \K'\to A\to 0$ with $\K'_f$ flat over $R$ and $A\in\textbf{T}$ supported on the central fiber.
\\

Indeed, write $0\to\K_f\to \K\to M\to 0$, where $\K_f\subset \K$ is the largest flat subsheaf of $\K$ over $R$ and $M$ is supported on the central fiber. 
Consider the short exact sequence 
$$0\to A\to M\to B\to 0$$ with $A\in\textbf{T}$ and $B\in\textbf{F}$. 
Define $\HH':=\text{ker}\,(\HH\to \HH_k\to B)$. The section $\mathcal{O}_{\mathcal{Y}}\xrightarrow{s} \HH$ factors through $\HH'$ with cokernel $\K'$. We have a short exact sequence $$0\to \K'\to \K\to B\to 0,$$ 
which implies that 
$0\to \K_f\to \K'\to A\to 0$. The subsheaf $\HH'$ has the desired property. 
\\

\textbf{Step 3.} We prove that the sheaf $\HH'$ from Step $2$ is in $\textbf{F}$. For this, suppose $\HH\in\textbf{F}$ and let $\HH'$ be defined by $$0\to\HH'\to\HH\to B\to 0,$$ where $B$ is supported on the central fiber in $\textbf{F}$. We claim that $\HH'_k\in\textbf{F}$. 

First, we have that $\HH'_K\cong\HH_K$. For the central fiber, consider the sequence \eqref{tor}:
$$0\to B\to\mathcal{H}'_k\to\mathcal{H}_k\to B\to 0.$$ 
We obtain short exact sequences
\begin{align*}
    0&\to P\to \HH_k\to B\to 0,\\
    0&\to B\to \mathcal{H}'_k\to P\to 0.
\end{align*}
The sheaf $P$ is a subsheaf of $\HH_k$, so $A\in\textbf{F}$. The extension of two sheaves in $\textbf{F}$ is in $\textbf{F}$, so $\mathcal{H}'_k\in\textbf{F}$. 

Replace $\HH$ with $\HH'$. 
This means that there exists a flat extension $\HH$ of $F$ over $R$ such that for
$\mathcal{I}:=\left[\mathcal{O}_\mathcal{Y}\xrightarrow{s}\mathcal{H}\right]$, we have $\mathcal{I}_K\cong I$,
$\HH_k\in\textbf{F}$, and the cokernel $\mathcal{K}:=\text{coker}(s)$ fits in a sequence
$0\to \K_f\to \K\to A\to 0$ with $\K_f$ flat over $R$ and $A\in\textbf{T}$ supported on the central fiber.
\\

\textbf{Step 4.} Next, we claim that there exists a subsheaf $\HH'\subset \HH$ such that the section $s$ factors through $\OO_{\mathcal{Y}}\xrightarrow{s'} \HH'$ and has cokernel $\mathcal{K}':=\text{coker}\,(s')$ in $\textbf{T}$. 

For a sheaf $C\in \text{Coh}_{\leq 1}(Y)$, let $F(C)\in\textbf{F}$ be the sheaf in $\textbf{F}$ from the torsion pair short exact sequence. 
First, a diagram chasing shows that $F(\K_f)=F(\K).$ Let $P_0:=F(\K_f)$.
We assume the claim is false.
Consider the short exact sequence:
\begin{equation}\label{ses}
0\to Q_0\to \K_{k}\to P_0\to 0.
\end{equation}
Let $\HH^1:=\text{ker}\,(\HH\to P_0).$ 
The section $\OO_{\mathcal{Y}}\xrightarrow{s} \HH$ factors through $\OO_{\mathcal{Y}}\to \HH^1$. After restricting this sequence to the central fiber, we obtain the short exact sequence
$$0\to P_0\to \HH^1_k\to A_0\to 0.$$ Let $\K^1:=\text{coker}\left(\OO_Y\to \HH^1_k\right)$, then we get two short exact sequences
\begin{align*}
    0&\to Q_1\to \K^1\to P_1\to 0,\\
    0&\to R_0\to \K^1\to Q_0\to 0,
\end{align*}
where the first is obtained as \eqref{ses} for $\K^1$ and
where $R_0$ is a quotient of $P_0$. 
We repeat this process to obtain sequences
\[0\to Q_n\to \K^n\to P_n\to 0\] and quotients $P_n\twoheadrightarrow R_n$. By the assumption that the claim is false, the sheaf $P_n$ is non-trivial. Observe that 
\[\mathrm{supp}(\mathcal{K}^{n+1})-\mathrm{supp}(\mathcal{K}^n)=\mathrm{supp}(R_n)-\mathrm{supp}(P_n)\in H_2(Y),\] 
where we let $\mathcal{K}^0:=\mathcal{K}_k$.
For $H$ ample line bundle on $Y$, the intersection of $H$ with the supports above is non-zero, and equals zero if the supports are zero. Thus, for $n$ large enough, we have $\mathrm{supp}(\mathcal{K}^{n+1})=\mathrm{supp}(\mathcal{K}^n)$ and $\mathrm{supp}(R_n)=\mathrm{supp}(P_n)$. If $R_n$ is not isomorphic to $P_n$, then $\mathrm{ker}(P_n\twoheadrightarrow R_n)$ has support of dimension zero, and thus $P_n$ cannot be in $\textbf{F}$. Thus, from $n$ large enough, we have that $P_n\cong R_n$. We may assume this happens from $n=0$.
Similarly to Step $1$, we show that the eventually $P_n\cong P_{n+1}$ for $n$ large enough. There are maps $Q_{n+1}\to Q_n$ for $n\geq 0$. 
The sequence for $P_n$ thus also becomes eventually constant; we assume this happens from $n=0$.

Using an argument as in Step $1$,
the quotient $\K_f/\pi\twoheadrightarrow P_0$ on the central fiber can be lifted to a quotient 
$$\K_f\twoheadrightarrow \mathcal{L},$$ where $\mathcal{L}$ is flat over $R$. We have $\mu(\mathcal{L}_k)=\mu(P_0)$. 
Consider the torsion pair sequence
$$0\to A\to \mathcal{L}_K\to B\to 0,$$ with $A\in \textbf{T}$ and $B\in \textbf{F}$. 
There is a surjection $(\K_f)_K\twoheadrightarrow B$. Further, we have that $(\K_f)_K\in \textbf{T}$. 
This means that $B=0$ and thus that $\mathcal{L}_K\in\textbf{T}$. 
We thus have that $\mu(\mathcal{L})=\mu(\mathcal{L}_0)<a\leq \mu(\mathcal{L}_K)=\mu(\mathcal{L})$.
This contradiction explains that there indeed exists a subsheaf $\HH'\subset \HH$ such that $\text{coker}\left(\OO_{\mathcal{Y}}\to \HH'\right)\in\textbf{T}$. 
\\

\textbf{Step 5.}
Using the argument in Step $3$, we see that $\HH'_k\in\textbf{F}$. 
Replace $\mathcal{H}'$ with $\HH$.
The $R$-flat complex $\mathcal{I}:=\left[\OO_{\mathcal{Y}}\xrightarrow{s}\HH\right]$ is thus a BS family and $\mathcal{I}_K\cong I$. 
\end{proof}

\subsection{The virtual fundamental class}\label{vfc}
Next, we explain that $\text{BS}^f_n(Y,\beta)$ has a natural virtual fundamental class. Proposition \ref{open} implies that the locus of BS pairs inside $\mathfrak{LM}$ is open. We now explain how the result follows from \cite{ht}.
Let $\mathcal{H}om\,(\mathbb{I},\mathbb{I})_0$ be the kernel of the trace map $$\text{Tr}: \mathcal{H}om\,(\mathbb{I},\mathbb{I})\to \OO_{Y\times \text{BS}^f_n(Y, \beta)}.$$
The Atiyah class $\text{At}\in \text{Ext}^1\left(\mathbb{I}, \mathbb{I}\otimes \mathbb{L}_{Y\times \text{BS}^f_n(Y, \beta)}\right)$ induces a map, see \cite[Sections 4.2, 4.5]{ht}:
\begin{equation}\label{equ}
    R\pi_{2*}\left(R\mathcal{H}om(\mathbb{I},\mathbb{I})_0\otimes\pi_1^*\omega_{Y}\right)[2]\to \mathbb{L}_{\text{BS}^f_n(Y, \beta)}.
\end{equation}

\begin{thm}\label{vir}
The map \eqref{equ} is a perfect obstruction theory. Thus the algebraic space $\text{BS}^f_n(Y,\beta)$ carries a virtual fundamental class
$$[\text{BS}^f_n(Y,\beta)]^{\text{vir}}\in A_d(\text{BS}^f_n(Y,\beta)), 
H_{2d}(\text{BS}^f_n(Y,\beta),\mathbb{Z}),$$
where $d=-\chi(\text{RHom}(I,I)_0)=\beta\cdot c_1(Y)$
for $I$ a BS pair in $\text{BS}^f_n(Y,\beta)$.
\end{thm}

\begin{proof}
The map \eqref{equ} is an obstruction theory by \cite[Theorem 4.1 and Section 4.5]{ht}. It is further perfect by the comment at the end of \cite[Section 4.3]{ht} and by Proposition \ref{aut}. The space $\text{BS}^f_n(Y,\beta)$ thus carries a virtual fundamental class of dimension $d$ by \cite{bf}.
\end{proof}

\subsection{BS generating series}\label{inse}
Let $\mathbb{I}\in D^b\left(Y\times \text{BS}^f_n(Y,\beta)\right)$ be the universal BS pair.  For an insertion $\gamma\in H^l(Y,\mathbb{Z})$ and an integer $k\geq 0$, define
$$\ch_{2+k}(\gamma)(-): H_*(\text{BS}^f_n(Y,\beta),\mathbb{Q})\to H_{*-2k+2-l}(\text{BS}^f_n(Y,\beta),\mathbb{Q})$$ by the formula
$$\ch_{2+k}(\gamma)(-)=\pi_{2*}\left(\text{ch}_{2+k}(\mathbb{I}) \pi_1^*(\gamma)\cap \pi_2^*(-)\right).$$
 The BS invariants with insertions $\gamma_1,\cdots,\gamma_r$ and descendant levels $\kappa_1,\cdots, \kappa_r\geq 0$ are defined by \[\langle \tau_{\kappa_1}(\gamma_1),\cdots, \tau_{\kappa_r}(\gamma_r)\rangle_{\beta,n}=\int_{\left[\text{BS}^f_n(Y,\beta)\right]^{\text{vir}}}\ch_{2+\kappa_1}(\gamma_1)\cdots \ch_{2+\kappa_r}(\gamma_r).\]
The generating series for BS invariants with insertions and descendant levels as above and class $\beta$ is given by the Laurent series in $q$: $$\text{BS}^f_{\beta}(q; \gamma, \kappa):=\sum_{n\in\mathbb{Z}} \langle \tau_{\kappa_1}(\gamma_1),\cdots,\tau_{\kappa_r}(\gamma_r)\rangle_{\beta,n} q^n.$$
The total generating series of BS invariants with insertions and descendant levels as above is defined by the generating series in $\mathbb{C}[\Delta]_{\Phi}$:
$$\text{BS}^f(q; \gamma, \kappa):=\sum_{\beta\in N_{ 1}(Y)}\text{BS}^f_{\beta}(q; \gamma, \kappa)\,q^\beta.$$

\section{Degeneration formula for BS invariants}\label{deg}

In this section, we define relative BS invariants and prove a degeneration formula for BS invariants following the degeneration formula for DT invariants of Li--Wu \cite{lw}. 
In Subsection \ref{rbsp}, we define relative BS pairs and show that their moduli are DM stacks. In Subsection \ref{degbs2}, we prove a degeneration formula at the cycle level; during this proof, we also construct  virtual fundamental classes for relative BS pairs as in Maulik--Pandharipande--Thomas \cite{mpt}, Li--Wu \cite{lw}.
In Subsection \ref{relBS}, we define relative BS invariants and state the degeneration formula.

\subsection{Relative BS pairs}\label{rbsp}
Let $X$ and $Y$ be projective threefolds with $Y$ is smooth, and let $f:Y\to X$ be a birational map with $Rf_*\mathcal{O}_Y=\mathcal{O}_X$. Let $U\subset X$ be the maximal open subset such that $f^{-1}(U)\xrightarrow{\sim} U$ is an isomorphism and let $E:=Y\setminus f^{-1}(U)$.
Consider $S\subset Y$ a divisor which does not intersect $E$. Recall the definition of $Y[k]$, the $k$-step degeneration of $Y$, from Subsection \ref{relPT}.
\textit{A relative BS pair} is a two term complex $$I=\left[\mathcal{O}_{Y[k]}\xrightarrow{s} F\right],$$ where $F$ is a sheaf on $Y[k]$ with $\pi_*[F]=\beta\in H_2(Y,\mathbb{Z})$, $\chi(F)=n$ such that:

(i) the restriction $I|_{Y\setminus S}$ is a BS pair for $f:Y\setminus S\to X$,

(ii) the restriction of $I|_{Y[k]\setminus E}$ is a PT relative pair, see Subsection \ref{pt} or \cite[Section 3.7]{pt} for definitions.

 Consider the functor \[\Phi_{\text{rBS}(Y,S)}:(\text{Schemes}/k)^{\text{op}}\to\text{Sets}\] with $\Phi_{\text{rBS}(Y,S)}(B)$ the set of equivalence classes of pairs 
$\OO_{Y[k]\times B}\to\mathcal{F}$ for some $k\geq 0$ such that $\OO_{Y[k]\times b}\to \mathcal{F}|_{Y[k]\times b}$ is a relative BS pair for every $b\in B$. Two families $\mathcal{F}$ and $\mathcal{F}'$ are equivalent if there exists a line bundle $\mathcal{L}$ on $B$ such that $\mathcal{F}\cong \mathcal{F}'\otimes\pi_2^*\mathcal{L}$.

\begin{thm}\label{prorel}
The functor $\Phi_{\text{rBS}(Y,S)}$ is represented by a proper DM stack $\text{BS}^f_n(Y/S,\beta)$.
\end{thm}

\begin{proof}
We need to show that the functor $\Phi_{\text{rBS}(Y,S)}$ is bounded, open, separated, complete, and has finite automorphisms. The functor $\Phi_{\text{rBS}(Y,S)}$ is a subfunctor of \[\Phi_{\text{rBS}(Y,S)}\hookrightarrow \Phi_{\text{BS}(Y)}\times\Phi_{\text{rPT}(Y\setminus E, S)}.\] 
It is immediate to see that this implies that $\Phi_{\text{rBS}(Y, S)}$ is bounded and open. Further, BS pairs on $Y$ have trivial automorphisms and relative PT pairs on $(Y\setminus E, S)$ have finite automorphisms, so $\Phi_{\text{rBS}(Y, S)}$ has finite automorphisms. Separatedness follows as in Proposition \ref{sep} using the analogue of Proposition \ref{aut}.

To show that $\Phi_{\text{rBS}(Y, S)}$ is complete, let $R$ be a DVR with fraction field $K$. Let $\mathcal{Y}=Y\times R$, $\mathcal{E}=E\times R$, and $\mathcal{S}=S\times R$.
Consider a $K$-relative BS pair \[I=\left[\OO_{\mathcal{Y}_K[k]}\to F\right].\]
There exists an extension of the BS pair $\OO_{\mathcal{Y}_K\setminus \mathcal{S}_K}\to F$ to an $R$-flat BS pair
\begin{equation}\label{c1}
    \mathcal{I}^1=\left[\OO_{\mathcal{Y}\setminus \mathcal{S}}\xrightarrow{s^1} \mathcal{F}^1\right].\end{equation}
Further, there exists an extension of the PT pair $I_{\mathcal{Y}_K}=
\left[\OO_{\mathcal{Y}_K[k]}
\to F_{\mathcal{Y}_K[k]}\right]$ to an $R$-flat PT pair for some $s\geq k$:
\begin{equation}\label{c2}
    \mathcal{I}^2=\left[\OO_{\mathcal{Y}[s]\setminus \mathcal{E}}\xrightarrow{s^2} \mathcal{F}^2\right].
    \end{equation}
    To see this, extend the pair trivially over $E$ and use completeness of the functor $\Phi_{\text{rPT}(Y, S)}$ \cite[Subsection 6.4]{lw}.
    Let $\mathcal{U}:=\mathcal{Y}\setminus \left(\mathcal{S}\cup \mathcal{E}\right)$.
The pairs $\mathcal{I}^1|_{\mathcal{U}}$ and $\mathcal{I}^2|_{\mathcal{U}}$ are PT pairs such that
\begin{equation}\label{isopt}
    \mathcal{I}^1|_{\mathcal{U}_K}\cong \mathcal{I}^2|_{\mathcal{U}_K}\cong I|_{\mathcal{U}_K}.\end{equation}
By the separatedness of the moduli of PT pairs, the isomorphism \eqref{isopt} extends to 
\[\mathcal{I}^1|_{\mathcal{U}}\cong \mathcal{I}^2|_{\mathcal{U}}.\]
We can glue the sheaves $\mathcal{F}^1$ and $\mathcal{F}^2$ and the sections $s^1$ and $s^2$ to obtain a complex $\mathcal{I}=\left[\mathcal{O}_\mathcal{Y}\xrightarrow{s}\mathcal{F}\right]$ with $\mathcal{I}|_{\mathcal{Y}\setminus \mathcal{S}}\cong \mathcal{I}_1$ and $\mathcal{I}|_{\mathcal{Y}[s]\setminus \mathcal{E}}\cong \mathcal{I}_2$. The complex $\mathcal{I}$ is the desired extension of $I$. 
\end{proof}

We next check that relative BS pairs are open in the derived category. 
\begin{prop}\label{operel}
Let $B_0$ be a scheme over $k$, and $i_0: B_0\hookrightarrow B$ a nilpotent thickening. Consider $$I_0=\left[\OO_{Y[k]\times B_0}\to F_0\right]$$
a relative BS pair over $B_0$. Let $I\in D^b\left(Y[k]\times B\right)$ be a complex with trivial determinant such that $Li_0^*I\cong I_0.$
Then there exists a sheaf $F$ on $Y[k]\times B$, flat over $B$, such that $$I\cong \left[\OO_{Y[k]\times B}\to F\right].$$
\end{prop}

\begin{proof}
The analogous statement for relative PT pairs follows as in \cite[Theorem 2.7]{pt}.
By Proposition \ref{open}, we have that
$$I|_{(Y\setminus S)\times B}\cong  \left[\OO_{(Y\setminus S)\times B}\xrightarrow{s^1} F^1\right],$$ 
where $F^1$ is flat over $B$.
Using the analogous statement for relative PT pairs, we have that
$$I|_{\left(Y[k]\setminus E\right)\times B}\cong I^2:= \left[\OO_{\left(Y[k]\setminus E\right)\times B}\xrightarrow{s^2} F^2\right],$$
where $F^2$ is flat over $B$.
The restrictions of $F^1$ and $F^2$ to $\left(Y\setminus \left(S\cup E\right)\right)\times B$ are isomorphic. The sections $s^1$ and $s^2$ are identified via this isomorphism, so $I$ is a relative BS pair.  
\end{proof}

\subsection{The cycle degeneration formula for BS invariants}\label{degbs2}
Let $B$ be a smooth curve and let $0\in B$.
Consider fourfolds $\mathcal{Y}$ and $\mathcal{X}$ fibered over $B$
\begin{equation*}
\begin{tikzcd}[column sep=small]
\mathcal{Y}\arrow[rr, "f"]\arrow[rd]& & \mathcal{X}\arrow[ld]\\
 & B &
\end{tikzcd}  
\end{equation*}
with $\mathcal{Y}$ a smooth fourfold and such that:
\begin{itemize}
    \item for $b\in B\setminus 0$, the fiber $\mathcal{Y}_b$ is smooth and
    the morphism $f_b: \mathcal{Y}_b\to \mathcal{X}_b$ is birational with $Rf_{b*}\mathcal{O}_{\mathcal{Y}_b}=\mathcal{O}_{\mathcal{X}_b}$, and
    \item the restriction over $0$ of the map $f$ is a morphism $$f_0: Y_1\cup_S Y_2\to X_1\cup_S X_2$$ such that $Y_1$ and $Y_2$ are smooth threefolds that intersect transversely in a smooth divisor $S$, $f_0|_{Y_1}: Y_1\cong X_1$, $X_1$ and $X_2$ intersect transversely in $S$, and $g:=f_0|_{Y_2}: Y_2\to X_2$ is a birational map with $Rf_{2*}\mathcal{O}_{Y_2}=\mathcal{O}_{X_2}$.
\end{itemize}
In particular, $S$ and the exceptional locus of $g: Y_2\to X_2$ do not intersect.
\\

We next discuss the setup for proving degeneration formulas and, in particular, define perfect obstruction theories for relative BS invariants following the proofs of degeneration formulas for DT and PT invariants \cite{lw}, \cite{mpt}. As in loc. cit., we obtain a cycle degeneration formula in Theorem \ref{cycle}.

Let $b\in B\setminus 0$ and let $(\beta,n)\in N_{\leq 1}(\mathcal{Y}_b)$. 
Let $\mathcal{B}:=\mathcal{B}(\beta, n)$ be the Artin stack of $(\beta,n)$-decorated semistable models of $\mathcal{Y}/B$ with universal family $$\widetilde{\mathcal{Y}}\to\mathcal{B}.$$
There is a natural map $\mathcal{B}\to B$.
Denote by $\mathcal{B}_0$ and $\widetilde{\mathcal{Y}}_0$ the fibers over $0$. We have that $\mathcal{B}\setminus\mathcal{B}_0\cong B\setminus 0$, and the universal family restricted to $\mathcal{B}\setminus\mathcal{B}_0$ is
$\mathcal{Y}\setminus\mathcal{Y}_0\to B\setminus 0.$
We can write \begin{equation}\label{zerof}
    \widetilde{\mathcal{Y}}_0=\bigcup_{k\geq 0} \mathcal{Y}_0[k],
    \end{equation}
    where $\mathcal{Y}_0[k]$ is defined by
$$\mathcal{Y}_0[k]:=Y_1\cup_S \mathbb{P}_S(N_S\oplus \OO)\cup_S\cdots\cup_S \mathbb{P}_S(N_S\oplus \OO)\cup_S Y_2,$$ with $k$ copies of $\mathbb{P}_S(N_S\oplus \OO)$, where $N_S$ is the normal bundle of $S$ in $Y_1$. The stack $\mathcal{B}$ has an étale map to the stack $\mathcal{C}$ from \cite[Subsection 2.1]{lw} where we ignore the decorations, and in particular it is smooth. The stack $\mathcal{C}_0$ is a smooth zero-dimensional stack, locally of finite type, consisting of a point for each degeneration level $k\geq 0$ with automorphism group $\mathbb{G}_m^k$, see \cite[Subsection 2.2]{Mnotes} for more details.

Let $\mathcal{P}\to\mathcal{B}$ be the stack of relative BS pairs on the fibers of $\widetilde{\mathcal{Y}}\to\mathcal{B}$. One can show as in Proposition \ref{prorel} that $\mathcal{P}$ is a proper DM stack with finite automorphism over $B$. Using the same argument as in Proposition \ref{operel} and the results of \cite{ht}, there is a relative perfect obstruction theory of $\mathcal{P}/\mathcal{B}$ defined as follows. Let $\mathbb{I}$ be the universal complex over $\widetilde{\mathcal{Y}}\times_{\mathcal{B}}\mathcal{P}$ and consider the relative perfect obstruction theory
\begin{equation}\label{o1}
    E:=R\pi_{2*}\left(R\mathcal{H}om(\mathbb{I}, \mathbb{I})_0\otimes \pi_1^*\omega_{\widetilde{\mathcal{Y}}/\mathcal{B}}\right)[2]\to \mathbb{L}_{\mathcal{P}/\mathcal{B}}.
    \end{equation}
We obtain a perfect obstruction theory \begin{equation}\label{o2}
    \mathcal{E}\to\mathbb{L}_{\mathcal{P}}
    \end{equation}
    from the relative perfect obstruction theory \eqref{o1} as follows. Consider the natural map $\phi: E\to \mathbb{L}_{\mathcal{P}/\mathcal{B}}\to \mathbb{L}_{\mathcal{B}}[1]$ and let $\mathcal{E}:=\text{cone}(\phi)[-1]$. The perfect obstruction theory \eqref{o2} is obtained from the following diagram:
\begin{equation*}
\begin{tikzcd}
\mathcal{E}\arrow{r}\arrow{d}& E\arrow{d}\arrow{r}& \mathbb{L}_{\mathcal{B}}[1]\arrow{d}{\text{id}}\\
\mathbb{L}_{\mathcal{P}}\arrow{r}& \mathbb{L}_{\mathcal{P}/\mathcal{B}}\arrow{r}& 
\mathbb{L}_{\mathcal{B}}[1].
\end{tikzcd}
\end{equation*}
Let $\mathcal{P}_0:=\mathcal{P}\times_B 0$.
Using the relative perfect obstruction theory \eqref{o1}, we obtain a perfect obstruction theory
\begin{equation}\label{o3}
    \mathcal{E}_0\to \mathbb{L}_{\mathcal{P}_0}
\end{equation} which fits in the diagram
\begin{equation*}
\begin{tikzcd}
\mathcal{E}_0\arrow{r}\arrow{d}& E|_{\mathcal{P}_0}\arrow{d}\arrow{r}& \mathbb{L}_{\mathcal{B}_0}[1]\arrow{d}{\text{id}}\\
\mathbb{L}_{\mathcal{P}_0}\arrow{r}& \mathbb{L}_{\mathcal{P}_0/\mathcal{B}_0}\arrow{r}& 
\mathbb{L}_{\mathcal{B}_0}[1].
\end{tikzcd}
\end{equation*}
Let $L_0$ be line bundle over $\mathcal{B}$ corresponding to $\mathcal{B}_0\subset \mathcal{B}$.
The perfect obstruction theories for $\mathcal{P}_0$ and $\mathcal{P}$ can be compared via the triangle, see \cite[Diagram 54]{mpt}:
\begin{equation}\label{unu}
    \mathcal{E}|_{\mathcal{P}_0}\to \mathcal{E}_0\to L^{\dual}_0[1].
\end{equation}
Next, denote by $\nu$ the data of two pairs $(\beta_1, n_1)$, $(\beta_2, n_2)$, and a positive integer $k$ such that $\beta_1+\beta_2=\beta$ and $n+k=n_1+n_2$. Then $S\cdot\beta_1=S\cdot\beta_2=k$, see \cite[Lemma 2.2]{hl}.
Define
$$\mathcal{P}_{\nu}:=\text{PT}_{n_1}(Y_1/S, \beta_1)\times_{\text{Hilb}(S,k)} \text{BS}^g_{n_2}(Y_2/S, \beta_2).$$
Then $\mathcal{P}_0=\bigcup_{\nu}\mathcal{P}_{\nu}$. Further, for every $\nu$, there exists a divisor \begin{equation}\label{bnu}
    \mathcal{B}_{\nu}\subset \mathcal{B}
\end{equation} whose pull-back to $\mathcal{P}$ is $\mathcal{P}_{\nu}$. Let $L_{\nu}$ be the associated line bundle to $\mathcal{B}_{\nu}$. Then $$\bigotimes_{\nu}L_{\nu}\cong L_0.$$ We can define a perfect obstruction theory $\mathcal{E}_{\nu}\to \mathbb{L}_{\mathcal{P}_{\nu}}$ as above which fits in a distinguished triangle:

\begin{equation}\label{doi}
    \mathcal{E}|_{\mathcal{P}_{\nu}}\to \mathcal{E}_{\nu}\to L^{\dual}_{\nu}[1].
\end{equation}
Let $\mathcal{P}_1=\text{PT}_{n_1}(Y_1/S, \beta_1)$ and $\mathcal{P}_2=\text{BS}^g_{n_2}(Y_2/S, \beta_2)$.
As in \eqref{o1}, there are relative perfect obstruction theories 
\begin{equation}\label{o54}
    E_i\to \mathbb{L}_{\mathcal{P}_i/\mathcal{B}_{\eta_i}}
\end{equation}
for $i\in\{1,2\}$, see \cite[Section 3.9]{mpt} for the definition of the stack $\mathcal{B}_{\eta_i}$. As in \eqref{o2}, there are perfect obstruction theories
\begin{equation}\label{o55}
    \mathcal{E}_i\to \mathbb{L}_{\mathcal{P}_i/\mathcal{B}_{\eta_i}}
\end{equation}
for $i\in\{1,2\}$ constructed from \eqref{o54}. 
The perfect obstruction theories \eqref{o55} induce natural virtual fundamental classes $\left[\text{PT}_{n_1}(Y_1/S,\beta_2)\right]^{\text{vir}}$ and $\left[\text{BS}^g_{n_2}(Y_2/S,\beta_2)\right]^{\text{vir}}$ with rational coefficients by \cite{bf}. 
Further, the perfect obstruction theories of $\mathcal{P}_{\nu}$ and its factors $\mathcal{P}_1$ and $\mathcal{P}_2$ are compared as follows, see \cite[Diagram 61]{mpt}:

\begin{equation}\label{trei}
  \begin{tikzcd}
  \mathcal{E}_1\oplus \mathcal{E}_2\arrow{r} \arrow{d}& \mathcal{E}_{\nu}\arrow{d} \arrow{r}& \Omega_{\text{Hilb}\,(S,k)}[1]\arrow{d}\\
  \mathbb{L}_{\mathcal{P}_1\times\mathcal{P}_2}\arrow{r}& \mathbb{L}_{\mathcal{P}_{\nu}}\arrow{r}& \mathbb{L}_{\mathcal{P}_{\nu}/\mathcal{P}_{1}\times\mathcal{P}_2}.
  \end{tikzcd}  
\end{equation}

\begin{thm}\label{cycle}
(a) Let $b\in B\setminus o$. Then $i_b^![\mathcal{P}]^{\text{vir}}=\left[\text{BS}^{f_b}_n(\mathcal{Y}_b, \beta)\right]^{\text{vir}}$.

(b) The restriction over $o$ has virtual fundamental class $i_0^![\mathcal{P}]^{\text{vir}}=[\mathcal{P}_0]^{\text{vir}}$.

(c) The fundamental class of the special fiber decomposes $[\mathcal{P}_0]^{\text{vir}}=\sum_{\nu}i_{\nu *}[\mathcal{P}_{\nu}]^{\text{vir}}$, where the sum is over all data $\nu$ consisting of two pairs $(\beta_1, n_1)$, $(\beta_2, n_2)$, and a positive integer $k$ such that $\beta_1+\beta_2=\beta$ and $n+k=n_1+n_2$. 

(d) The fundamental class $[\mathcal{P}_{\nu}]^{\text{vir}}$ further decomposes in the contributions of the two factors as follows
$$[\mathcal{P}_{\nu}]^{\text{vir}}=\Delta^!
\left([\text{PT}_{n_1}(Y_1/S,\beta_1)]^{\text{vir}}\times \left[\text{BS}^g_{n_2}(Y_2/S,\beta_2)\right]^{\text{vir}}\right).$$
Here, $\Delta$ is the diagonal embedding $\Delta:\text{Hilb}(S,k)\hookrightarrow \text{Hilb}(S,k)\times \text{Hilb}(S,k)$.
\end{thm}

\begin{proof}
(a) The restriction of the perfect obstruction theory of $\mathcal{P}$ to the fiber over $b$ is the same as the perfect obstruction theory for 
$\text{BS}^{f_b}_n(\mathcal{Y}_b, \beta)$, see \cite[Proposition 6.2]{lw}.

(b) Use the triangle \eqref{unu}, see \cite[Equation 6.6]{lw}.

(c) Use the triangle \eqref{doi} and $L_0\cong\otimes_{\nu} L_{\nu}$, the proof follows as in \cite[Proposition 6.4]{lw}.

(d) The claim follows from the diagram \eqref{trei}, see the proof in the DT case given in \cite[Proposition 6.5]{lw}.
\end{proof}



\subsection{Relative BS invariants}\label{relBS}
Recall the map $g :Y_2\to X_2$ from Subsection \ref{degbs2}. 
Recall the definitions of cohomologically weighted partition, of the basis $C_{\eta}$ for $H^{\cdot}(\text{Hilb}(S,\beta), \mathbb{Q})$, and of the invariants $|\eta|, l(\eta), \xi(\eta)$ from Subsection \ref{pt}. 
Consider the morphism $$\varepsilon: BS^g_n(Y_2/S,\beta)\to \text{Hilb}(S,d)$$
which sends a sheaf $F$ to its restriction to $S$, where $S\cdot\beta=d$. 
The relative BS invariants for insertions $\gamma_1,\cdots,\gamma_r\in H^{\cdot}(Y_2,\mathbb{Z})$ and descendant levels $\kappa_1,\cdots, \kappa_r\geq 0$ are defined by the formula:
\[\langle \tau_{\kappa_1}(\gamma_1),\cdots, \tau_{\kappa_r}(\gamma_r)|\,\eta\rangle^{\text{BS}}_{\beta,n}=\int_{\left[\text{BS}^g_n(Y_2/S,\beta)\right]^{\text{vir}}}\ch_{2+\kappa_1}(\gamma_1)\cdots \ch_{2+\kappa_r}(\gamma_r)\cap \varepsilon^*(C_{\eta}).\]
The generating series for the relative BS invariants with insertions and descendant levels as above, support $\beta$, and cohomologically weighted partition $\eta$ is defined by:
$$\text{BS}^{g/S}_{\beta, \eta}(q; \gamma, \kappa):=\sum_{n\in\mathbb{Z}} \langle \tau_{\kappa_1}(\gamma_1),\cdots,\tau_{\kappa_r}(\gamma_r)|\,\eta\rangle^{\text{BS}}_{\beta,n}\, q^n.$$

\begin{thm}\label{degen}
Consider a family $f:\mathcal{Y}\to\mathcal{X}$ as in Subsection \ref{degbs2}, consider a point $b\in B\setminus o$, denote by $f_b:\mathcal{Y}_b\to\mathcal{X}_b$, and let $\beta\in N_1\left(\mathcal{Y}_b\right)$. Consider insertion classes $\gamma_1,\cdots,\gamma_r\in H^{\cdot}(\mathcal{Y},\mathbb{Z})$ and descendant levels $\kappa_1,\cdots, \kappa_r\geq 0$. We abuse notation and write $\gamma_1, \cdots, \gamma_r$ for the restriction of the insertions classes to $\mathcal{Y}_b$, $Y_1$, and $Y_2$.
Then
$$\text{BS}^{f_b}_{\beta}(q; \gamma, \kappa)=
\sum \text{PT}^{Y_1/S}_{\beta_1,\eta}(q; \gamma, \kappa)
\text{BS}^{g/S}_{\beta_2,\eta^{\dual}}(q; \gamma, \kappa)\frac{(-1)^{|\eta|-l(\eta)}\xi(\eta)}{q^{|\eta|}},$$ 
where the sum on the right hand side is taken over all splittings $\beta_1+\beta_2=\beta$ and over all cohomologically weighted partitions $\eta$.
\end{thm}

\begin{proof}
The degeneration formula follows from its cycle version in Theorem \ref{cycle} as in the DT case \cite[Theorem 6.8]{lw}.
\end{proof}

\section{The BS/ PT correspondence}\label{wc}

Recall the statement of Theorem \ref{thm2}. For this, fix $f:Y\to X$ a contraction of a curve $C\cong \mathbb{P}^1$ with normal bundle $N_{C/Y}\cong\mathcal{O}_C(-1)^{\oplus 2}.$
We will use the notations
\begin{align*}
    \mathbb{P}&:=\mathbb{P}_C\left(\OO(-1)^{\oplus 2}\oplus \OO\right),\\
    S&:=\mathbb{P}_C\left(\OO(-1)^{\oplus 2}\right),\\
    N&:=\mathbb{P}\setminus S=\text{Tot}_C\left(\OO(-1)^{\oplus 2}\right).
\end{align*}
Let $g:\mathbb{P}\to \mathbb{P}'$ be the contraction of the zero section $C$.

The plan for this Section is as follows. In Subsection \ref{sss}, we describe the $T$-fixed BS and PT pairs on $\mathbb{P}$ and show that the moduli spaces $\text{BS}_n(\pi, m)$ and $\text{PT}_n(\pi, m)$ of $T$-fixed BS or PT pairs on $\mathbb{P}$ intersecting $S$ transversely has a \textit{symmetric} perfect obstruction theory. 
In Subsections \ref{lrbspt} and \ref{lrbspt2}, 
we reduce computations for relative BS or PT invariants on $\mathbb{P}$ relative to $S$ to computations on the spaces $\text{BS}_n(\pi, m)$ and $\text{PT}_n(\pi, m)$. 
We then use the degeneration formulas for BS and PT invariants and the virtual localization formula to reduce the proof of Theorem \ref{thm2} to a wall-crossing statement between invariants of $\text{BS}_n(\pi, m)$ and $\text{PT}_n(\pi, m)$.

We will be using localization for BS pairs in the current Section. In \cite[Subsections 2.3 and 2.4]{hl}, Henry Liu analyzed in great detail the fixed BS pairs for the particular geometry $f:S\times\mathbb{C}\to \left(\mathbb{C}^2/\Gamma\right)\times\mathbb{C}$, where the quotient $\mathbb{C}^2/\Gamma$ has a type $A$ singularity. The argument in our case is more formal. In particular, we do not need such a detailed description of $T$-fixed BS pairs.


\subsection{Localization for BS and PT pairs}
\label{sss} Let $\left(\mathbb{C}^*\right)^3$ be the torus acting naturally on $\mathbb{P}$, and let $T\cong \left(\mathbb{C}^*\right)^2 \subset\left(\mathbb{C}^*\right)^3$ be the subtorus which preserves the natural Calabi-Yau form on $N$. We describe the $T$-fixed BS and PT pairs on $\mathbb{P}$.

The $T$-invariant points are two points $0,\infty\in C$ and the $T$-invariant points on $S$. The $T$-invariants lines not contained in $S$ are $C$ and the two legs different from $C$ from each of $0$ and $\infty$. Call these legs $L_i$ for $1\leq i\leq 4$. 
The torus $T$ also acts naturally on $\text{BS}^g_n\left(\mathbb{P}, \beta\right)$, $\text{BS}^g_n\left(\mathbb{P}/S, \beta\right)$, and on the analogous PT moduli spaces.

Let $I=\left[\OO_{\mathbb{P}}\xrightarrow{s} F\right]$ be a $T$-fixed BS or PT pair on $\mathbb{P}$ which intersects the divisor $S\subset \mathbb{P}$ transversely. 
The restriction of $I$ to $Y\setminus C$ is an ideal sheaf on a toric variety which intersects $S$ transversely. The ideal sheaf $I|_{Y\setminus C}$ corresponds to $T$-fixed ideals $\pi_i\subset \mathbb{C}[x,y]$, so to ideals $\pi_i$ generated by monomials for every leg $1\leq i\leq 4$. 
Let $h\in H_2(\mathbb{P},\mathbb{Z})$ be the tautological class on $\mathbb{P}$. It is the class of the legs $L_i$ for $1\leq i\leq 4$.
We use the notations
\begin{align*}
    \ell(\pi_i)&=\text{dim}_\mathbb{C}\,\mathbb{C}[x,y]/\pi_i,\\
    [\pi_i]&=\ell(\pi_i)h\in H_2(\mathbb{P}),\\
    [\pi]&=\sum_{i=1}^4 [\pi_i]\in H_2(\mathbb{P}).
\end{align*}
For such a partition profile $\pi=\left(\pi_i\right)_{i=1}^4$ and an integer $m\geq 0$, let $$\text{PT}_n(\pi,m)\subset \text{PT}_n\left(\mathbb{P}, [\pi]+m[C]\right),\, \text{PT}_n\left(\mathbb{P}/S, [\pi]+m[C]\right)$$ be the subspace of $T$-fixed PT pairs $\left[\OO_{\mathbb{P}}\xrightarrow{s} F\right]$ with partition profile $\pi$ which intersect $S$ transversely. Define similarly \[\text{BS}_n(\pi, m)\subset \text{BS}^g_n\left(\mathbb{P}, [\pi]+m[C]\right), \text{BS}^g_n\left(\mathbb{P}/S, [\pi]+m[C]\right).\]
The main result we prove in this Subsection is:

\begin{prop}\label{prof}
The subspaces $\text{BS}_n(\pi,m)$ and $\text{PT}_n(\pi,m)$ are proper algebraic spaces with natural symmetric perfect obstruction theories constructed 
from the perfect obstruction theories for $\text{BS}^g_n(\mathbb{P}, [\pi]+m[C])$ and $\text{PT}_n(\mathbb{P}, [\pi]+m[C])$, respectively.
\end{prop}

In particular, $\text{BS}_n(\pi,m)$ and $\text{PT}_n(\pi,m)$ have virtual dimension zero.
Before we start the proof of Proposition \ref{prof}, we need the following preliminary result:

\begin{lemma}\label{halp}

Let $I=\left[\mathcal{O}_\mathbb{P}\xrightarrow{s} F\right]$ be a $T$-fixed complex intersecting $S$ transversely and such that $\text{Hom}(I, I)^T\cong\mathbb{C}$. Consider a map $\omega_{\mathbb{P}}\to \mathcal{O}_{\mathbb{P}}$. Then there are isomorphisms
\[\text{Ext}^k_{\mathbb{P}}(I, I\otimes \omega_{\mathbb{P}})^T\cong \text{Ext}^k_{\mathbb{P}}(I, I)^T\] for $k\in\{1, 2\}$.
\end{lemma}

\begin{proof}
We check the statement for $k=1$, the statement for $k=2$ follows from Serre duality.
Consider the triangle \eqref{tri2} for $I$:
\[\mathcal{I}_{D}\to I\to Q[-1]\rightarrow \mathcal{I}_{D}[1].\]
We claim that it suffices to show that
\begin{equation}\label{isoI}
    \text{Ext}^k_{\mathbb{P}}(E_1,E_2\otimes \omega_{\mathbb{P}})^T\cong \text{Ext}^k_{\mathbb{P}}(E_1,E_2)^T\text{ for }k=1,2,
\end{equation}
where $E_i$ is either $\mathcal{I}_{D}$ or $Q[-1]$. We have a long exact sequence 
\[0\to \text{Ext}^1(I, \mathcal{I}_D\otimes \omega)\to \text{Ext}^1(I, I\otimes\omega)\to \text{Ext}^1(I, Q[-1]\otimes\omega)\to \text{Ext}^2(I, \mathcal{I}_D\otimes \omega)\] and a similar one for $\text{Ext}^1(I, I)$. We have that 
\[\text{Ext}^0(I, Q\otimes\omega)\cong \text{Ext}^0(I, Q)\] because the support of $Q$ is disjoint from $S$. It thus suffices to show that 
\begin{equation*}
    \text{Ext}^k(I, \mathcal{I}_D\otimes\omega)^T\cong \text{Ext}^k(I, \mathcal{I}_D)^T\text{ for }k=1,2.
\end{equation*}
Using again two long exact sequences and comparing them, the isomorphism above for $k=2$ follows directly from \eqref{isoI} and the isomorphism for $k=1$ follows from \eqref{isoI}
and from the isomorphism
\begin{equation}\label{isotorusid}
\text{Hom}(I, \mathcal{I}_D)^T\cong \text{Hom}(\mathcal{I}_D, \mathcal{I}_D)^T\cong\mathbb{C}.
\end{equation} The isomorphism \eqref{isotorusid} follows from the exact sequence
\[0=\text{Hom}(I, Q[-2])^T\to \text{Hom}(I, \mathcal{I}_D)^T\to \text{Hom}(I, I)^T\cong\mathbb{C}\to \text{Hom}(I, Q[-1])^T=0.\]
We are thus left with showing \eqref{isoI}.
The canonical divisor $\omega_{\mathbb{P}}$ is supported on $S$, so there are isomorphisms
\begin{align*}
    \text{Hom}_{\mathbb{P}}\left(\mathcal{I}_{D}, Q\otimes\omega_{\mathbb{P}}\right)&\cong \text{Hom}_{\mathbb{P}}\left(\mathcal{I}_{D}, Q\right)\\
    \text{Ext}^2_{\mathbb{P}}\left(Q, \mathcal{I}_{D}\otimes\omega_{\mathbb{P}}\right)&\cong \text{Ext}^2_{\mathbb{P}}\left(Q, \mathcal{I}_{D}\right)\\
    \text{Ext}^1_{\mathbb{P}}\left(Q, Q\otimes\omega_{\mathbb{P}}\right)&\cong \text{Ext}^1_{\mathbb{P}}\left(Q, Q\right).
\end{align*}
We next need to show that 
\begin{equation}\label{isoII}
    \text{Ext}^1_{\mathbb{P}}(\mathcal{I}_{D},\mathcal{I}_{D}\otimes \omega_{\mathbb{P}})^T\cong \text{Ext}^1_{\mathbb{P}}(\mathcal{I}_{D},\mathcal{I}_{D})^T.
\end{equation}
There is a spectral sequence with terms 
\begin{equation}\label{spectralsequence}
    H^i\left(\mathbb{P}, \mathcal{E}xt^j(\mathcal{I}_D, \mathcal{I}_D\otimes \mathcal{L})\right)
\end{equation} with $i, j\geq 0$ and $i+j=1$ converging to $\text{Ext}^1_{\mathbb{P}}(\mathcal{I}_{D},\mathcal{I}_{D}\otimes\mathcal{L})$ for $\mathcal{L}$ equal to $\mathcal{O}_\mathbb{P}$ or $\omega_\mathbb{P}$. 
We claim that
\begin{equation}\label{isoIII}
    \text{Ext}^1_{\mathbb{P}\setminus C}(\mathcal{I}_{D\setminus C},\mathcal{I}_{D\setminus C}\otimes \omega_{\mathbb{P}})^T\cong \text{Ext}^1_{\mathbb{P}\setminus C}(\mathcal{I}_{D\setminus C},\mathcal{I}_{D\setminus C})^T.
    \end{equation}
We first explain that, for fixed $D$, \eqref{isoII} is true if and only if \eqref{isoIII} is true. It suffices to show
\[H^i_C\left(\mathbb{P}, \mathcal{E}xt^j(\mathcal{I}_D, \mathcal{I}_D)\right)^T\cong H^i_C\left(\mathbb{P}, \mathcal{E}xt^j(\mathcal{I}_D, \mathcal{I}_D\otimes \omega_{\mathbb{P}})\right)^T\] for $i, j\geq 0$ and $i+j=1$. This is true because $\omega_{\mathbb{P}}$ is supported on $S$, so the sheaves $\mathcal{E}xt^j(\mathcal{I}_D, \mathcal{I}_D)$ and $\mathcal{E}xt^j(\mathcal{I}_D, \mathcal{I}_D\otimes \omega_{\mathbb{P}})$ are isomorphic near $C$.

The curve $D\setminus C$ has connected components corresponding to the legs $(L_i)_{i=1}^4$. To show \eqref{isoIII}, it suffices to show that \eqref{isoII} holds for $D$ a Cohen-Macaulay curve supported on one leg $L_i$. In this case, 
\begin{align}\label{zeroTinvariants}
\text{Ext}^1_{\mathbb{P}}(\mathcal{I}_{D},\mathcal{I}_{D})^T&=0,\\
\text{Ext}^2_{\mathbb{P}}(\mathcal{I}_{D},\mathcal{I}_{D})^T&\cong
\text{Ext}^1_{\mathbb{P}}(\mathcal{I}_{D},\mathcal{I}_{D}\otimes \omega_{\mathbb{P}})^T=0.
\end{align}
Both vanishing follow from \cite[Lemmas 6 and 8]{mnop1}; the results in loc. cit. are applicable for the torus $T\subset \left(\mathbb{C}^*\right)^3$ for curves $D$ determined only by a partition $\pi_i$ along a leg $L_i$. 
\end{proof}

\begin{proof}[Proof of Proposition \ref{prof}]
We need to show that the $T$-fixed perfect obstruction theories for $\text{BS}_n(\pi,m)$ and $\text{PT}_n(\pi,m)$ are symmetric. Let $I$ be a BS or a PT pair. Then $\text{Hom}(I, I)^T\cong\mathbb{C}$, see Proposition \ref{aut} or \cite[Lemma 1.15]{pt}.
The perfect obstruction theory for $\text{BS}^g_n(\mathbb{P}, [\pi]+m[C])$ or $\text{PT}_n(\mathbb{P}, [\pi]+m[C])$ restricts over $I$ to $\left[E^{-1}_{I}\to E^0_{I}\right]$, with cohomology $h^{-1}=\text{Ext}^1(I,I)$ and $h^0=\text{Ext}^2(I,I)$.
The perfect obstruction theory on the $T$-fixed locus restricts over $I$ to $$\left[\left(E_I^{-1}\right)^{T}\to \left(E_I^{0}\right)^{T}\right]$$
with cohomology $h^{-1}=\text{Ext}_{\mathbb{P}}^1(I,I)^{T}$ and $h^0=\text{Ext}_{\mathbb{P}}^2(I,I)^{T}$. 
By Proposition \ref{halp} and Serre duality, we obtain a nondegenerate pairing:
\begin{multline*}
    \text{Ext}^1_{\mathbb{P}}(I,I)^T\times 
\text{Ext}^2_{\mathbb{P}}(I,I)^T\xrightarrow{\sim}
\text{Ext}^1_{\mathbb{P}}(I,I)^T\times 
\text{Ext}^2_{\mathbb{P}}(I,I\otimes \omega_{\mathbb{P}})^T \to\\ \text{Ext}^3_{\mathbb{P}}(I_1,I_1\otimes\omega_{\mathbb{P}})^T \xrightarrow{\sim}\mathbb{C}.
\end{multline*}
The last isomorphism follows from Proposition \ref{aut} and Serre duality.
A similar argument in the global case shows the existence of a non-degenerate pairing on $E^T=\left[\left(E^{-1}\right)^T\to \left(E^0\right)^T\right]$. 
\end{proof}

\subsection{Localization for relative BS and PT pairs I}\label{lrbspt}
Let $I=\left[\mathcal{O}_{\mathbb{P}[k]}\xrightarrow{s} F\right]$ be a $T$-fixed relative BS complex. Then the ideal sheaf of the support of $F$ is transversal to $S$ and has a given partition profile $\pi=(\pi_i)_{i=1}^4$ along the legs $(L_i)_{i=1}^4$. 
Denote by 
\begin{align*}
    \text{BS}^g_n(\mathbb{P}/S, \pi, m)&\subset \text{BS}^g_n(\mathbb{P}/S,[\pi]+m[C])^T\\
    \text{PT}_n(\mathbb{P}/S, \pi, m)&\subset \text{PT}_n(\mathbb{P}/S,[\pi]+m[C])^T
\end{align*} the subsets whose restrictions to the the legs $(L_i)_{i=1}^4$ have partition profile $\pi$. Using an argument as in Proposition \ref{prof}, these spaces have natural symmetric perfect obstruction theories obtained by localization. 
The cokernel of $s$ is supported on $0, \infty\in C$ or beyond $\mathbb{P}$. The spaces $\text{BS}^g_n(\mathbb{P}/S, \pi, m)$ and $\text{PT}_n(\mathbb{P}/S, \pi, m)$ have connected components depending on the behaviour of $I$ at $0, \infty$, on the partition profile $(\pi_i)_{i=1}^4$ along the legs $(L_i)_{i=1}^4$, and on the behaviour beyond $\mathbb{P}$.
Given a profile partition $\pi$, let $D_{\pi}\subset \mathbb{P}$ be the Cohen-Macaulay curve supported on $\bigcup_{i=1}^4 L_i$ with partition profile $\pi$. Denote by $\chi(\pi):=\chi(\mathcal{O}_{D_\pi})$.
Let \[\text{PT}_n\left(\left(\mathbb{P}\setminus C\right)/S, \pi, 0\right)\subset \text{PT}_n\left(\mathbb{P}/S, \pi, 0\right)\] be the union of connected components of $\text{PT}_n\left(\mathbb{P}/S, \pi, 0\right)$ such that $I|_{\mathbb{P}\setminus S}\cong \mathcal{I}_{D_{\pi}\setminus S}$. 
Then $\text{BS}^g_n(\mathbb{P}/S, \pi, m)$ has a decomposition in subsets $\text{BS}^g_n(\mathbb{P}/S, \pi, m)_j$ indexed by $j\geq 0$, all but finitely many empty, which are union of connected components of $\text{BS}^g_n(\mathbb{P}/S, \pi, m)$ and
such that
\begin{equation}\label{decoBS}
    \text{BS}^g_n(\mathbb{P}/S, \pi, m)_j\cong \text{BS}_{n-j}(\pi, m)\times \text{PT}_{\chi(\pi)+j}\left(\left(\mathbb{P}\setminus C\right)/S, \pi, 0\right).
\end{equation}
The analogous analysis holds for PT pairs, so $\text{PT}_n(\mathbb{P}/S, \pi, m)$ has a decomposition in subsets $\text{PT}_n(\mathbb{P}/S, \pi, m)_j$ indexed by $j\geq 0$, all but finitely many empty, which are union of connected components of $\text{PT}_n(\mathbb{P}/S, \pi, m)$ and
such that
\begin{equation}\label{decoPT}
    \text{PT}_n(\mathbb{P}/S, \pi, m)_j\cong \text{PT}_{n-j}(\pi, m)\times \text{PT}_{\chi(\pi)+j}\left(\left(\mathbb{P}\setminus C\right)/S, \pi, 0\right).
\end{equation}

\begin{prop}\label{prop:decomposition}
The decompositions \eqref{decoBS} and \eqref{decoPT} are compatible with the virtual fundamental classes of the spaces involved.
\end{prop}

\begin{proof}
We discuss the statement for \eqref{decoPT}. The proof is similar to the proof of Lemma \eqref{halp}.
Denote the obstruction theories by $E^\cdot_{\mathbb{P}/S}$, $E^\cdot_{\pi}$, and $E^\cdot_{\mathbb{P}\setminus C/S}$. 
By restriction, there are natural maps $E^\cdot_{\mathbb{P}/S}\to E^\cdot_{\mathbb{P}\setminus C/S}$ and $E^\cdot_{\mathbb{P}/S}\to E^\cdot_{\pi}$. We explain that
\[\text{Ext}^1_{\mathbb{P}/S}(I, I)^T\xrightarrow{\sim} \text{Ext}^1_{\mathbb{P}\setminus S}(I, I)^T\oplus \text{Ext}^1_{\mathbb{P}\setminus C/S}(I, I)^T,\]
the analogous relations for trace-free $\text{Ext}^0_0$ and $\text{Ext}^3_0$ are immediate and the one for $\text{Ext}^2$ is similar.
Using the Ext spectral sequence, see \eqref{spectralsequence}, it suffices to check that 
\[\mathcal{E}xt^1_{\mathbb{P}/S}(I, I)^T\xrightarrow{\sim} \mathcal{E}xt^1_{\mathbb{P}\setminus S}(I, I)^T\oplus \mathcal{E}xt^1_{\mathbb{P}\setminus C/S}(I, I)^T.\] This is true because the restriction of $\mathcal{E}xt^1_{\mathbb{P}/S}(I, I)^T$ to $\mathbb{P}\setminus (C\cup S)$ is zero, see \eqref{zeroTinvariants}, \cite[Proof of Lemma 6]{mnop1}. 
\end{proof}

From Proposition \ref{prop:decomposition} we obtain that:
\begin{prop}\label{vircomp}
Let $d$ be the virtual dimension of $\text{BS}^g_n(\mathbb{P}, \beta)$ and let $j\geq 0$. Then
\[
\left[\text{BS}^g_n(\mathbb{P}/S,\pi,m)_j^T\right]^{\text{vir}}=\left[\text{PT}_{\chi(\pi)+j}\left(\left(\mathbb{P}\setminus C\right)/S, \pi, 0\right)\right]^{\text{vir}}\times
  \left[\text{BS}_{n-j}(\pi,m)\right]^{\text{vir}}\] in $H_{2d}\left(\text{BS}^g_n(\mathbb{P}/S,\pi,m)_j,\mathbb{Q}\right)$. Similarly, we have that
\[\left[\text{PT}_n(\mathbb{P}/S,\pi,m)_j^T\right]^{\text{vir}}=
  \left[\text{PT}_{\chi(\pi)+j}\left(\left(\mathbb{P}\setminus C\right)/S, \pi, 0\right)\right]^{\text{vir}}\times
  \left[\text{PT}_{n-j}(\pi,m)\right]^{\text{vir}}\] in
  $H_{2d}\left(\text{PT}_n(\mathbb{P}/S,\pi,m)_j,\mathbb{Q}\right).$
\end{prop}

Recall the definition of the generating series for relative invariants from Subsections \ref{relPT} and \ref{relBS}. 
Let $\eta$ be a cohomologically weighted partition.

We next state Propositions \ref{releq} and \ref{releqf}. 
In Proposition \ref{propa}, we show that Theorem \ref{halln}, to be proved in Section \ref{hall}, implies these two Propositions and Theorem \ref{thm2}.

\begin{prop}\label{releq}
The following equality holds:
\begin{equation*}
    \text{BS}^{g/S}_{ \eta}(q)=\frac{\text{PT}^{\mathbb{P}/S}_{ \eta}(q)}{\text{PT}^{\text{exc}}(q)},
\end{equation*}
where the generating series have no insertions.
\end{prop}

Let $m, k\geq 0$, $n\in\mathbb{Z}$, and $\pi=(\pi_i)_{i=1}^4$ be a partition profile. Denote by $N^{\text{relBS}, \text{vir}}_{n, \pi, m, j}$ the virtual normal bundle of
\[\text{BS}^g_n(\mathbb{P}/S, \pi, m)_j\subset \text{BS}^g_n\left(\mathbb{P}/S, [\pi]+m[C]\right).\]
Define the series
\begin{align*}
\text{BS}^{g/S}_{\pi, m, j, \eta}(q)&:=\sum_{n\in\mathbb{Z}}
\left(\int_{\left[\text{BS}^g_n(\mathbb{P}/S, \pi, m)_j\right]^{\text{vir}}}\frac{\varepsilon^*(C_\eta)}
{e\left(N^{\text{relBS}, \text{vir}}_{n, \pi, m, j}\right)}\right) q^n\\
\text{BS}^{g/S}_{\pi, j, \eta}(q)&:=\sum_{m\geq 0} \text{BS}^{g/S}_{\pi, m, j, \eta}(q)\, q^{m[C]}.
\end{align*}
We define similarly the virtual normal bundle $N^{\text{relPT}, \text{vir}}_{n, \pi, m, j}$ and the series $\text{PT}^{\mathbb{P}/S}_{\pi, m, j, \eta}$, $\text{PT}^{\mathbb{P}/S}_{\pi, j, \eta}$. 

\begin{prop}\label{releqf}
The following equality holds:
\begin{equation*}
    \text{BS}^{g/S}_{\pi, j, \eta}(q)=\frac{\text{PT}^{\mathbb{P}/S}_{\pi, j, \eta}(q)}{\text{PT}^{\text{exc}}(q)}.
\end{equation*}
\end{prop}

Define the formal series 
\begin{align*}
\text{BS}_{\pi, m}(q)&:=\sum_{n\in\mathbb{Z}}\left(\int_{\left[\text{BS}_n(\pi, m)\right]^{\text{vir}}}\frac{1}
{e\left(N^{\text{BS}, \text{vir}}_{n, \pi, m}\right)}\right) q^n\\
\text{BS}_{\pi, \eta}(q)&:=\sum_{m\geq 0} \text{BS}_{\pi, m, \eta}(q)\, q^{m[C]}.
\end{align*}

\begin{prop}\label{nee}
Fix $\pi$, $\eta$, $j$ as above. There exists a constant $b\in\mathbb{Q}$ such that
\begin{align*}
    \text{BS}^{g/S}_{\pi, j, \eta}(q)&=bq^{-j}\text{BS}_{\pi, \eta}(q),\\
    \text{PT}^{\mathbb{P}/S}_{\pi, j, \eta}(q)&=bq^{-j}\text{PT}_{\pi, \eta}(q).
\end{align*}
\end{prop}

\begin{proof}
Recall the setting of \eqref{o55}, let $\mathcal{T}:=\mathcal{B}_{\eta_i}$, consider the relative perfect obstruction theory 
\begin{equation}\label{o56}
    E:=R\pi_{2*}(R\mathcal{H}om(\mathbb{I},\mathbb{I})_0\otimes\pi_1^*\omega_{\mathbb{P}})[2]\to \mathbb{L}_{\text{BS}^g_n(\mathbb{P}/S, \beta)/\mathcal{T}},
\end{equation}
and recall the construction of the perfect obstruction theory for $\text{BS}^g_n(\mathbb{P}/S, \beta)$ from \eqref{o55}:
\begin{equation*}
\begin{tikzcd}
\mathcal{E}\arrow{r}\arrow{d}& E\arrow{d}\arrow{r}& \mathbb{L}_{\mathcal{T}}[1]\arrow{d}{\text{id}}\\
\mathbb{L}_{\text{BS}^g_n(\mathbb{P}/S, \beta)}\arrow{r}& \mathbb{L}_{\text{BS}^g_n(\mathbb{P}/S, \beta)/\mathcal{T}}\arrow{r}& 
\mathbb{L}_{\mathcal{T}}[1].
\end{tikzcd}
\end{equation*}
Let $N[1]$ be the cone of the map
\[R\pi_{2*}(R\mathcal{H}om(\mathbb{I},\mathbb{I})_0\otimes\pi_1^*\omega_{\mathbb{P}})[2]
\to 
R\pi_{2*}(R\mathcal{H}om(\mathbb{I}|_{\mathbb{P}},\mathbb{I}|_{\mathbb{P}})_0\otimes\pi_1^*\omega_{\mathbb{P}})[2]\] and let $\mathcal{N}$ be defined by
\begin{equation*}
\begin{tikzcd}
\mathcal{N}\arrow{r}\arrow{d}& N\arrow{d}\arrow{r}& \mathbb{L}_{\mathcal{T}}[1]\arrow{d}{\text{id}}\\
\mathbb{L}_{\text{BS}^g_n(\mathbb{P}/S, \beta)}\arrow{r}& \mathbb{L}_{\text{BS}^g_n(\mathbb{P}/S, \beta)/\mathcal{T}}\arrow{r}& 
\mathbb{L}_{\mathcal{T}}[1].
\end{tikzcd}
\end{equation*}
Denote by $M^{\text{BS}}_{n, \pi, m, j}$ the restriction of $\mathcal{N}$ to $\text{BS}^g_n(\mathbb{P}/S, \pi, m)_j$. It depends only on its behaviour beyond $C$. Denote by $M'^{\text{BS}}_{n, \pi, m, j}$ the sum of non-zero weight subspaces of $M^{\text{BS}}_{n, \pi, m, j}$.
Then 
\begin{equation}\label{BSprime}
    e\left(N^{\text{relBS}, \text{vir}}_{n, \pi, m, j}\right)=e\left(N^{\text{BS}, \text{vir}}_{n, \pi, m, j}\right)e\left(M'^{\text{BS}}_{n, \pi, m, j}\right)
\end{equation} and so 
\begin{multline}\label{BSprime2}
    \int_{[\text{BS}_n(\mathbb{P}/S, \pi, m)_j]^{\text{vir}}}\frac{\varepsilon^*(C_\eta)}
    {e\left(N^{\text{relBS}, \text{vir}}_{n, \pi, m, j}\right)}=\\
\int_{[\text{PT}_{\chi(\pi)+j}(\mathbb{P}/S, \pi, 0)]^{\text{vir}}}\frac{\varepsilon^*(C_\eta)}
{e\left(M'^{\text{BS}}_{n, \pi, m, j}\right)}
\int_{[\text{BS}_{n-j}(\pi, m)]^{\text{vir}}}\frac{1}{e\left(N^{\text{BS}, \text{vir}}_{n-j, \pi, m}\right)}.
\end{multline}
Define similarly $M^{\text{PT}}_{n, \pi, m, j}$ and $M'^{\text{PT}}_{n, \pi, m, j}$. Then there are analogous results for PT invariants to \eqref{BSprime} and \eqref{BSprime2}. The conclusion follows from $M'^{\text{PT}}_{n, \pi, m, j}\cong M'^{\text{BS}}_{n, \pi, m, j}$. 
\end{proof}

\subsection{Localization for relative BS and PT pairs II}\label{lrbspt2}


Let $\mathbb{C}^*\subset T$ be a generic torus.
For a $\mathbb{C}^*$-representation $V$, denote by $V_w$ its weight $w$-subspace and let $V^m:=\bigoplus_{w\in\mathbb{Z}\setminus\{0\}}V_w.$
For $A, B\in D^b(\mathbb{P})$ two $T$-fixed complexes, the vector spaces $\text{Ext}^i(A,B)$ are natural $\mathbb{C}^*$-representations. 
For $w$ a weight, let $\text{ext}^i(A,B)_w:=\dim \text{Ext}^i(A,B)_w$, and denote by:
\begin{align*}
\text{ext}^i(A,B)^+&:=\sum_{w\in\mathbb{N}\setminus 0} \text{ext}^i(A,B)_w,\\
    \chi(A,B)^+&:=\sum_{i=0}^3 (-1)^i \text{ext}^i(A,B)^+.
\end{align*} 

Let $B_{\pi, m, n}$ be the set of connected components of $\text{BS}_n(\pi, m)$ and let $P_{\pi, m, n}$ be the set of connected components of $\text{PT}_n(\pi, m)$. Denote by $D=D_{\pi}$ the Cohen-Macaulay curve supported on $\bigcup_{i=1}^4 L_i$ with partition profile $\pi$. 
For $k$ in $B_{\pi, m, n}$ or $P_{\pi, m, n}$, consider a complex 
\[I=\left[\mathcal{O}_{\mathbb{P}}\xrightarrow{s} F\right]\] that intersects $S$ transversely in the connected component corresponding to $k$. Define
$\ell(k):=\chi(I,I)^{+}-\chi(\mathcal{I}_D, \mathcal{I}_D)^{+}.$ Consider the class
\[e_\pi:=\frac{e\left(\text{Ext}^1(\mathcal{I}_D, \mathcal{I}_D)^m\right)}{e\left(\text{Ext}^2(\mathcal{I}_D, \mathcal{I}_D)^m\right)}.\]

\begin{prop}\label{nee2}
For $k\in B_{\pi, m, n}$, there exists a constant $\gamma:=\gamma(k)$ such that
\begin{align*}
    \int_{[\text{BS}_{n}(\pi, m)_k]^{\text{vir}}}\frac{1}{e\left(N^{\text{BS}, \text{vir}}_{n, \pi, m}\right)}&=(-1)^{\ell(k)}\int_{[\text{BS}_{n}(\pi, m)_k]^{\text{vir}}}\frac{1}{e_\pi}=(-1)^{\ell(k)}\gamma\int_{[\text{BS}_{n}(\pi, m)_k]^{\text{vir}}}1,\\
    \int_{[\text{PT}_{n}(\pi, m)_k]^{\text{vir}}}\frac{1}{e\left(N^{\text{PT}, \text{vir}}_{n, \pi, m}\right)}&=(-1)^{\ell(k)}\int_{[\text{PT}_{n}(\pi, m)_k]^{\text{vir}}}\frac{1}{e_\pi}=(-1)^{\ell(k)}\gamma\int_{[\text{PT}_{n}(\pi, m)_k]^{\text{vir}}}1.
\end{align*}
\end{prop}

\begin{proof}
We discuss the equality for BS spaces.
We need to check that 
\begin{equation}\label{e432}
    e\left(N^{\text{BS}, \text{vir}}_{n, \pi, m}\right)=(-1)^{\ell(k)}e_\pi.
    \end{equation}
    Consider the triangles 
\begin{align*}
    h^0(I)&\to I\to h^1(I)[-1]\rightarrow h^0(I)[1],\\
    h^0(I)&\to \mathcal{I}_{D}\to J\rightarrow h^0(I)[1].
\end{align*}
Let $A, B$ be sheaves in the set $\{\mathcal{I}_D, J, h^1(I)\}$ such that not both of them are $\mathcal{I}_D$. To show \eqref{e432}, it suffices to show that
\begin{equation}\label{e4321}
    e\left(\sum_{i=0}^3 (-1)^i\text{Ext}^i(A, B)^m\right)=(-1)^{\chi(A,B)^{+}}.
\end{equation}
One of $A$ and $B$ is supported on $C$ and $\omega_{\mathbb{P}}|_C$ is trivial. By Serre duality, \[\text{Ext}^i(A, B)_w\cong \left(\text{Ext}^{3-i}(B, A)_{-w}\right)^{\vee}\] for any $w\in\mathbb{Z}$. We thus have
\[e\left(\text{Ext}^i(A, B)_w-\text{Ext}^{3-i}(B, A)_{-w}\right)=(-1)^{\text{ext}^i(A,B)_w},\]
and \eqref{e4321} follows. The constant $\gamma$ is the same for both BS and PT spaces because the integrand $e_\pi$ does not depend on the space used and the virtual classes $[\text{BS}_{n}(\pi, m)_k]^{\text{vir}}$ and $[\text{PT}_{n}(\pi, m)_k]^{\text{vir}}$ are of dimension zero.
\end{proof}

Fix a partition profile $\pi$ and $m\geq 0$. 
Define
\begin{align*}
   \text{BS}'_{n}(\pi, m)_k&=(-1)^{\ell(k)}\int_{[\text{BS}_n(\pi,m)_k]^{\text{vir}}}1,\\
   \text{PT}'_{n}(\pi, m)_k&=(-1)^{\ell(k)}\int_{[\text{PT}_n(\pi,m)_k]^{\text{vir}}}1
\end{align*}
for $k$ in $B_{\pi, m, n}$ or $P_{\pi, m, n}$, respectively. 
Define the generating series
\begin{align*}
    \text{BS}_{\pi, m}(q)&:=\sum_{n\in\mathbb{Z}} \sum_{k\in B_{\pi, m, n}} \text{BS}'_n(\pi,m)_k\,q^n,\\
    \text{PT}_{\pi, m}(q)&:=\sum_{n\in\mathbb{Z}} \sum_{k\in P_{\pi, m, n}} \text{PT}'_n(\pi,m)_k\,q^n.
\end{align*}
They are both Laurent series in $q$. Define the generating series
\begin{align*}
    \text{BS}_{\pi}(q, z)&:=\sum_{m\geq 0} \text{BS}_{\pi, m}(q)\,z^{m[C]},\\
    \text{PT}_{\pi}(q, z)&:=\sum_{m\geq 0} \text{PT}_{\pi, m}(q)\,z^{m[C]}.
\end{align*} 
They are both elements of $\mathbb{C}[\Delta]_{\Phi}$. Note that using the localization formula, see also Proposition \ref{inte} part (c), we have that
\begin{equation}\label{llpt}
    \text{PT}^{\mathbb{P}, \text{exc}}_{m[C]}(q)=\text{PT}_{0, m}(q),
\end{equation}
where $0$ is the zero partition profile.
In Section \ref{hall}, we prove the following wall-crossing result:

\begin{thm}\label{halln}
Let $\pi$ be a partition profile. Then
\[ BS_{\pi}(q, z)=\frac{PT_{\pi}(q, z)}{PT_0(q, z)}.\]
\end{thm}

We now explain that:

\begin{prop}\label{propa}
Theorem \ref{halln} implies Propositions \ref{releq} and \ref{releqf} and Theorem \ref{thm2}.
\end{prop}

\begin{proof}

\textbf{Step 1.} Theorem \ref{halln} for $q=z$ implies Proposition \ref{releqf} using Propositions \ref{nee} and \ref{nee2}.
\smallbreak

\textbf{Step 2.} The localization formula for the action of $T$ on $\mathbb{P}$ implies that
$$\text{BS}^{g/S}_{\beta,\eta}(q)=\sum_{\pi, m, j} \text{BS}^{g/S}_{\pi, m, j, \eta}(q)$$
where the sum is taken over partitions $\pi=(\pi_i)_{i=1}^4$, $m\geq 0$ such that $[\pi]+m[C]=\beta\in H_2(\mathbb{P})$, and $j\geq 0$.
The analogous result holds for PT invariants, and so Proposition \ref{releqf} implies Proposition \ref{releq}.
\smallbreak

\textbf{Step 3.} We explain that Proposition \ref{releq} implies Theorem \ref{thm2}. Let $p$ be the singular point of $X$. 
Consider the family
\begin{equation}\label{fami}
    \widetilde{f}: \text{Bl}_{C\times 0} (Y\times \mathbb{A}^1_{\mathbb{C}})\to \text{Bl}_{p\times 0}(X\times \mathbb{A}^1_{\mathbb{C}})
    \end{equation}
    over $\mathbb{A}^1_{\mathbb{C}}$.
For $t\neq 0,$ the fiber is the original contraction $f: Y\to X$. For $t=0$, the fiber is 
$$\text{Bl}_CY\cup \mathbb{P}\to \text{Bl}_p X\cup \mathbb{P}'.$$ The map $\text{Bl}_CY\xrightarrow{\sim} \text{Bl}_pX$ is an isomorphism, and the map $$g: \mathbb{P}\to \mathbb{P}'$$ is the contraction of the zero section $C$. 

 Let $p: \text{Bl}_{p\times 0}(X\times \mathbb{A}^1_{\mathbb{C}})\to X$ be the natural projection. The insertions are in $\widetilde{f}^*p^*H^i(X,\mathbb{Z})$. We may assume that the insertions are in degree $i>0$. When using the degeneration formulas for PT or BS invariants for the family \eqref{fami}, there will be no insertions in the $\mathbb{P}$ factors. Indeed, the cohomology of $H^\cdot(\mathbb{P})$ is generated by $[C]$ and $h$, and these classes do not intersect classes in $X$ of positive codimension.
 By the degeneration formula for PT invariants, see Subsection \ref{pt}, we have that
\begin{equation}\label{degPT}
\text{PT}^Y_{\beta}(q; \gamma, \kappa)=
\sum \text{PT}^{\text{Bl}_C(Y)/S}_{\beta_1, \eta}(q; \gamma, \kappa)\,
\text{PT}^{\mathbb{P}/S}_{\beta_2, \eta^{\dual}}(q)
\frac{(-1)^{|\eta|-l(\eta)}\xi(\eta)}{q^{|\eta|}},
\end{equation} where the sum is over all splittings of the curve class $\beta_1+\beta_2=\beta$ and over all cohomologically weighted partitions $\eta$. 
By the degeneration formula for BS invariants, see Theorem \ref{degen}, we have that
\begin{equation}\label{degBS}
\text{BS}^f_{\beta}(q; \gamma, \kappa)=
\sum \text{PT}^{\text{Bl}_C(Y)/S}_{\beta_1, \eta}(q; \gamma, \kappa)\,
\text{BS}^{g/S}_{\beta_2, \eta^{\dual}}(q)
\frac{(-1)^{|\eta|-l(\eta)}\xi(\eta)}{q^{|\eta|}},
\end{equation} where once again the sum is over all splittings of the curve class $\beta_1+\beta_2=\beta$ and over all cohomologically weighted partitions $\eta$. 
The decompositions \eqref{degPT} and \eqref{degBS} together with the Proposition \ref{releq} and Claim \eqref{llpt} imply Theorem \ref{thm2}.
\end{proof}

\section{The Hall algebra argument}\label{hall}

Let $W$ be a toric Calabi-Yau $3$-fold which contains a curve $C\cong \mathbb{P}^1$ with normal bundle $N\cong\mathcal{O}_C(-1)^{\oplus 2}$ such that the four legs from the $\left(\mathbb{C}^*\right)^3$-fixed points $0, \infty\in C$ different from $C$ are proper. We abuse notation and denote the four legs by $(L_i)_{i=1}^4$.
Let $V:=W\setminus N$. Let $h:W\to W'$ be the contraction of $C$.

In this section, we prove Theorem \ref{halln}. We revisit the argument of Bryan--Steinberg \cite{bs} which compares BS and DT/PT invariants for $h$. 


\subsection{Preliminaries.}
\subsubsection{} 
In this Subsection, we recall various Hall algebras and the integration map from the semi-continuous Hall algebra. For $(\beta, n)\in N_{\leq 1}(W)$, let $\mathfrak{M}_{\beta, n}$ be the moduli stack of sheaves on $W$ of compact support $(\beta, n)$. Let \[\mathfrak{M}:=\bigsqcup_{(\beta, n)\in N_{\leq 1}(W)}\mathfrak{M}_{\beta, n}.\] 
Recall the definitions of the motivic Hall algebra and of the Grothendieck ring of varieties from Subsection \ref{moHA}.
Consider a triplet $(\pi, m, n)$ as in Subsection \ref{sss}. 
Let $D_\pi\subset W$ be the Cohen Macaulay curve supported on the legs $\bigcup_{i=1}^4 L_i$ with partition profile $\pi$. 
Let \[\mathfrak{N}_{\pi, m, n}\subset \mathfrak{M}\] be the substack of sheaves $F$ with \begin{equation}\label{eq123}
    F|_{W\setminus C}\cong \mathcal{O}_{D_\pi}|_{W\setminus C},
    \end{equation}
    $\chi(F)=n$, and $[F]=m[C]+[\pi]\in H_2(W, \mathbb{Z})$. Then $F$ intersects $V$ transversely.
Denote by $\mathfrak{M}_{\beta, n}(\mathcal{O})$ the stack of sheaves $F$ of compact support $\beta$ and $\chi(F)=n$ with a section. 
Let \[\mathfrak{N}_\pi:=\bigsqcup_{\substack{n\in\mathbb{Z}\\m\geq 0}}\mathfrak{N}_{\pi, m, n}.\]
We define $\mathfrak{M}(\mathcal{O})$, $\mathfrak{N}_{\pi, m, n}(\mathcal{O})$ etc. similarly. 

\begin{defn}
Let $K\left(\text{Var}/k\right)_{\text{loc}}:=K\left(\text{Var}/k\right)\left[\mathbb{L}^{-1}, \left(1+\cdots+\mathbb{L}^n\right)^{-1}: n\geq 1\right]$.

Let
$H_{\text{reg}}\subset H:=K\left(\text{St}/\mathfrak{M}\right)$ be the $K\left(\text{Var}/k\right)_{\text{loc}}$ module generated by the span of classes $\left[V\xrightarrow{f}\mathfrak{M}\right]$ for $V$ a variety. One can assume that $V$ is an algebraic space and obtain the same module by \eqref{groth}. 
\end{defn}

\subsubsection{}
The following proposition is proved as in \cite{b}:

\begin{prop}
(a) The convolution product preserves $H_{\text{reg}}$, and thus $H_{\text{reg}}$ is a $K\left(\text{Var}/k\right)_{\text{loc}}$-algebra. 

(b) Let $H_{\text{sc}}:=H_{\text{reg}}/(\mathbb{L}-1)H_{\text{reg}}$. Then $H_{\text{sc}}$ is a commutative $K\left(\text{Var}/k\right)_{\text{loc}}$-algebra.  
\end{prop}

Recall the definition of the algebra
 $\mathbb{C}[\Delta]$ from Subsection \ref{Laurent}. Consider the Poisson algebra $\mathbb{C}[\Delta]$ with trivial Poisson bracket. 
Define the integration map \begin{equation}\label{smHA}
    \Psi: H_{\text{sc}}\to\mathbb{C}[\Delta]
    \end{equation}
    by the following formula, where $V$ is an algebraic space:
$$\Psi\left(\left[V\xrightarrow{f} \mathfrak{M}_{\beta,n}\right]\right)=\chi(V,f^*\nu)\,q^{\beta+n}.$$


\begin{thm}\label{ssf}
Consider a sheaf $F$ with compact support on $W$.
Let $G\subset \text{Aut}(F)$ be a maximal reductive group. The group $G$ acts naturally on $\text{Ext}^1_{W}(F, F)$.
There exist an open $G$-invariant set $F\in \mathcal{U}\subset \mathfrak{M}$, an open $G$-invariant subset $V\subset \text{Ext}^1_W(F, F)$, a holomorphic function $f:V/G\to\mathbb{A}^1_\mathbb{C}$, and a natural smooth map \begin{equation}\label{Phiphi}
    \Theta:\text{crit}(f)/G\to \mathcal{U}
\end{equation}
of relative dimension $\text{dim}\,\text{Aut}(F)-\text{dim}\,G$.
\end{thm}

\begin{proof}
The above statement is known for moduli stacks of (complexes of) sheaves on proper varieties, see Joyce-Song \cite[Theorem 5.5]{js}, Toda \cite[Corollary 5.7]{t}. Toda's proof also works for sheaves of compact support on not necessarily proper varieties. 
\end{proof}

\begin{thm}\label{poi}
The integration map $\Psi:H_{\text{sc}}\to\mathbb{C}[\Delta]$ defined above is a Poisson algebra homomorphism. 
\end{thm}

\begin{proof}
Theorem \ref{ssf} and the discussion in Subsection \ref{behrend} imply that for a sheaf $F$ with compact support,
the Behrend function can be computed by
$$\nu(F)=(-1)^{\text{dim}\,\text{Ext}^1_{W}(F, F)}\left(1-\chi(M_f)\right),$$ where $G$, $V$, and $f:V/G\to\mathbb{C}$ are as in Theorem \ref{ssf} and $M_f$ is the Milnor fiber of $f$ at any point in $\Theta^{-1}(I)$. The conclusion follows using localization techniques, see \cite[Theorem 5.11]{js} and \cite[Theorem 5.14]{js}.
\end{proof}

\subsubsection{} The Hall algebras $K(\text{St}/\mathfrak{M})$, $H_{\text{reg}}$, $H_{\text{sc}}$, and the Poisson algebra $\mathbb{C}[\Delta]$ are naturally $\Delta$-graded. 
We consider the corresponding Laurent algebras $K(\text{St}/\mathfrak{M})_{\Phi}$, $H_{\text{reg},\Phi}$, $H_{\text{sc},\Phi}$, and $\mathbb{C}[\Delta]_{\Phi}$.
The integration map $$\Psi: H_{\text{sc}}\to \mathbb{C}[\Delta]$$ extends to a continuous integration map
$$\Psi: H_{\text{sc},\Phi}\to \mathbb{C}[\Delta]_{\Phi}.$$

\begin{defn}
A morphism of stacks $f: \mathfrak{B}\to \mathfrak{M}$ is called \textit{$\Phi$-finite} if for every $(\beta, n)\in N_{\leq 1}(W)$,
the substack $\mathfrak{B}_{\beta,n}=f^{-1}(\mathfrak{N}_{\beta,n})$ is an Artin stack of finite type, and there exists a Laurent subset $S\subset \Delta$ such that $\mathfrak{B}_{\beta,n}$ is non-empty unless $(\beta,n)\in S$. 
\end{defn}

A $\Phi$-finite morphism $f: \mathfrak{B}\to \mathfrak{M}$ defines
an element 
$$\sum_{(\beta, n)\in S}[\mathfrak{B}_{\beta, n}\to\mathfrak{M}]$$
of $K(\text{St}/\mathfrak{M})_{\Phi}$ and of the other Hall algebras.

\subsubsection{}
Recall the stability condition $\mu: \text{Coh}_{\leq 1}(W)\to S$ from Subsection \ref{defs}. 
For each $s\in S$, there are torsion pairs $(\textbf{T}_s, \textbf{F}_s)$ 
of 
$\text{Coh}_{\leq 1}(W)$ 
constructed as in Subsection \ref{stabi}. 
For $a=(\infty, 0)$, the pair $(\textbf{T}_a, \textbf{F}_a)$ is the torsion pair used in the definition of BS pairs; for $b=(\infty, \infty)$, the pair $(\textbf{T}_b, \textbf{F}_b)$ is the torsion pair used in the definition of PT pairs. Recall the definition of $\text{SS}(I)$ from Subsection \ref{stabi}.

\begin{defn}\label{deff}
(a) For $s\in S$, define $\text{PT}^s_{\beta, n}\subset \mathfrak{M}_{\beta, n}(\mathcal{O})$ as the locus of complexes $\left[\OO_{W}\xrightarrow{s} F\right]$ with $[F]=(\beta, n)\in N_{\leq 1}(W)$, $F\in \textbf{F}_s$, and $\text{coker}(s)\in \textbf{T}_s$. 

Define $\text{PT}^s_{\pi, m, n}\subset \text{PT}^s_{\beta, n}$ as the locus of complexes $F$ in $\mathfrak{N}_{\pi, m, n}(\mathcal{O})$ for $\beta=m[C]+[\pi]\in H_2(W)$.  Let 
\begin{align*}
    \text{PT}^s&=
    \bigsqcup_{\beta, n}
    \text{PT}^s_{\beta, n},\\
    \text{PT}^s_{\pi}&=
    \bigsqcup_{m, n}\text{PT}^s_{\pi, m, n}.
\end{align*}

(b) If $I\subset S$ an interval, denote by $1_{\text{SS}(I)}$ the sublocus of $\mathfrak{M}(\mathcal{O})$ represented by sheaves in $\text{SS}(I)$, and by $1_{\text{SS}(I)}^{\OO}$ the sublocus of $\mathfrak{M}(\mathcal{O})$ represented by complexes $\left[\OO_{W}\xrightarrow{s} F\right]$ with $F$ in $\text{SS}(I)$.

(c) Define $\text{DT}_{\pi, m, n}\subset \mathfrak{M}_{\pi, m, n}(\mathcal{O})$ as the locus of complexes $\left[\OO_{W}\xrightarrow{s} F\right]$ with $s$ surjective and $F$ in $\mathfrak{N}_{\pi, m, n}$. For $s\in S$, define $\text{DT}^s_{\pi, m, n}\subset \text{DT}_{\pi, m, n}$ as the locus with $F\in \textbf{F}_s$. Let 
\begin{align*}
    \text{DT}_\pi&=
    \bigsqcup_{m, n}
    \text{DT}_{\pi, m, n},\\
    \text{DT}^s_{\pi}&=
    \bigsqcup_{m, n}\text{DT}^s_{\pi, m, n}.
\end{align*}
Let $\text{DT}_{\text{tors}}\subset \mathfrak{M}(\mathcal{O})$ be the locus of complexes $\left[\OO_{W}\xrightarrow{s} F\right]$ with $s$ surjective
and $F$ zero dimensional.
\end{defn}

\subsection{Identities in the Hall algebra}
In this Subsection, we prove Theorem \ref{halln} using the restriction to $\mathfrak{N}_\pi$ of the identities in the Hall algebra $K(\text{St}/\mathfrak{M})$
established by Bryan--Steinberg \cite{bs}. We revisit their argument and explain how to modify it to obtain Theorem \ref{halln}.

\subsubsection{} The following is \cite[Lemmas 67 and 68]{bs}.

\begin{prop}
(a) Let $s\in \{a, b\}$. Then
$\left[\text{PT}^s_{\pi} \to\mathfrak{M}\right]$,
$
\left[\text{DT}_{\pi}\to\mathfrak{M}\right]$,
$\left[\text{DT}^s_{\pi} \to\mathfrak{M}\right]$
are elements of the Hall algebra $H_{\Phi}$. 

(b) Let $I\subset S$ be an interval bounded from below. Then 
$\left[1_{\text{SS}(I)}\to \mathfrak{M}\right]$ is an element of
of the Hall algebra $H_{\Phi}$.
\end{prop}


We next relate generating series of BS and PT invariants from Theorem \ref{halln} to elements in $H_{\text{sc},\Phi}$.

\begin{prop}\label{inte} Denote by $z=q^{[C]}$ and abuse notation and write $q=q^1$ for the generators of $\mathbb{C}[\Delta]_{\Phi}$.
Consider the integration map $\Psi: H_{\text{sc},\Phi}\to \mathbb{C}[\Delta]_{\Phi}$. Then there is a constant $d_{\pi}$ depending only on $\pi$ such that

(a) $\Psi\left(\left[\text{PT}_{\pi}^a\to\mathfrak{M}\right]\right)=(-1)^{d_{\pi}}
\text{BS}_{\pi}(-q, z)$,

(b)
$\Psi\left(\left[\text{PT}_{\pi}^b\to\mathfrak{M}\right]\right)=(-1)^{d_{\pi}}\text{PT}_{\pi}(-q, z)$,

(c) $\Psi\left(\left[\text{PT}_{0}^b\to\mathfrak{M}\right]\right)=\text{PT}_{0}(-q, z)$.
\end{prop}

\begin{proof} 
Let $m\geq 0$, $n\in\mathbb{Z}$, and $\beta=m[C]+[\pi]\in H_2(W, \mathbb{Z})$. 
The spaces $\text{PT}^s_{\beta, n}$ are $T$-invariant for $s\in\{a, b\}$.
It suffices to show that there is a constant $d_\pi$ depending only on $\pi$ such that
\begin{align*}
    \Psi\left(\left[\text{PT}^a_{\pi, m, n}\to\mathfrak{M}\right]\right)&=(-1)^{d_\pi+n}\text{BS}'_n(\pi, m),\\
    \Psi\left(\left[\text{PT}^b_{\pi, m, n}\to\mathfrak{M}\right]\right)&=(-1)^{d_\pi+n}\text{PT}'_n(\pi, m).
\end{align*}
Using \cite[Lemma 26]{bs}, it suffices to show that
\begin{align*}
    \Psi\left(\left[\text{PT}^a_{\pi, m, n}\to\mathfrak{M}(\mathcal{O})\right]\right)&=(-1)^{d_\pi}\text{BS}'_n(\pi, m),\\
    \Psi\left(\left[\text{PT}^b_{\pi, m, n}\to\mathfrak{M}(\mathcal{O})\right]\right)&=(-1)^{d_\pi}\text{PT}'_n(\pi, m).
\end{align*}
Denote by $B_{\beta, n}$ the set of connected components of the $T$-fixed locus of $\text{BS}_n(W, \beta)$. For $k$ in $B_{\beta, n}$, let
$\varepsilon(k):=\chi(I,I)^+-\chi(\mathcal{O}, \mathcal{O})^+=\dim \text{Ext}^1(I, I)^m$ for $I$ a $T$-fixed BS complex in the connected component indexed by $k$, see Subsection \ref{lrbspt2} for the notation. 
Define similarly $P_{\beta, n}$ and $\ell(k)$ for $k$ in $P_{\beta, n}$. 
We have that
\begin{align*}
    \Psi\left(\left[\text{PT}^a_{\beta, n}\to\mathfrak{M}(\mathcal{O})\right]\right)&=\int_{\text{BS}_n(W, \beta)}\nu d\chi=\sum_{k\in B_{\beta, n}} (-1)^{\varepsilon(k)}\left[\text{BS}_n(W, \beta)_k\right]^{\text{vir}},\\
    \Psi\left(\left[\text{PT}^b_{\beta, n}\to\mathfrak{M}(\mathcal{O})\right]\right)&=\int_{\text{PT}_n(W, \beta)}\nu d\chi=\sum_{k\in P_{\beta, n}}(-1)^{\varepsilon(k)}\left[\text{PT}_n(W, \beta)_k\right]^{\text{vir}}.
\end{align*}
Indeed, the first equalities follow from Subsection \ref{behrend} and the second from the localization formula and from the analogue of Proposition \ref{nee2} in this context. 

Denote by $B'_{\beta, n}$ the subset of $B_{\beta, n}$ of connected components corresponding to BS complexes $\left[\mathcal{O}\xrightarrow{s} F\right]$ for $F$ in $\mathfrak{N}_{\pi, m, n}$. Let $B^o_{\beta, n}=B_{\beta, n}\setminus B'_{\beta, n}$. Define similarly $P'_{\beta, n}$ and $P^o_{\beta, n}$. By Subsection \ref{behrend} and the localization formula, we have that 
\begin{align*}
    \Psi\left(\left[\text{PT}^a_{\beta, n}\setminus \text{PT}^a_{\pi, m, n}\to\mathfrak{M}(\mathcal{O})\setminus \mathfrak{M}_{\pi}(\mathcal{O})\right]\right)&=\sum_{k\in B^o_{\beta, n}} (-1)^{\varepsilon(k)}\left[\text{BS}_n(W, \beta)_k\right]^{\text{vir}},\\
    \Psi\left(\left[\text{PT}^b_{\beta, n}\setminus \text{PT}^b_{\pi, m, n} \to\mathfrak{M}(\mathcal{O})\setminus \mathfrak{M}_{\pi}(\mathcal{O})\right]\right)&=\sum_{k\in P^o_{\beta, n}}(-1)^{\varepsilon(k)}\left[\text{PT}_n(W, \beta)_k\right]^{\text{vir}}.
\end{align*} The sets $B'_{\beta, n}$ and $B_{\pi, m, n}$ are naturally isomorphic; identify them via this isomorphism. For $k\in B_{\pi, m, n}$, we have that $\text{BS}_n(W,\beta)_k\cong \text{BS}_n(\pi, m)_k$ because they both parameterize complexes $[\mathcal{O}\to \mathcal{F}]$ intersecting the complement of $N$ transversely and satisfy the same hypotheses when restricted to $N$. Further, their perfect obstruction theories are naturally isomorphic and thus $\left[\text{BS}_n(W,\beta)_k\right]^{\text{vir}}=\left[\text{BS}_n(\pi, m)_k\right]^{\text{vir}}$.
We thus obtain that
\begin{align*}
    \Psi\left(\left[\text{PT}^a_{\pi, m, n}\to\mathfrak{M}(\mathcal{O})\right]\right)&=\sum_{k\in B'_{\beta, n}}(-1)^{\varepsilon(k)}\text{BS}'_n(\pi, m)_k,\\
    \Psi\left(\left[\text{PT}^b_{\pi, m, n}\to\mathfrak{M}(\mathcal{O})\right]\right)&=\sum_{k\in P'_{\beta, n}}(-1)^{\varepsilon(k)}\text{PT}'_n(\pi, m)_k.
\end{align*}
For $k$ in $B_{\pi, m, n}$, we have that $\varepsilon(k)=\ell(k)+\chi\left(\mathcal{I}_{D_\pi}, \mathcal{I}_{D_\pi}\right)^+-\chi(\mathcal{O}, \mathcal{O})^+=
\ell(k)+\dim \text{Ext}^1\left(\mathcal{I}_{D_\pi}, \mathcal{I}_{D_\pi}\right)^m$, see the argument using Serre duality from Subsection \ref{lrbspt2} for the second equality, and thus parts (a) and (b) follow for $d_\pi=\dim \text{Ext}^1\left(\mathcal{I}_{D_\pi}, \mathcal{I}_{D_\pi}\right)^m$. When $\pi=0$, we have that $d_0=\dim \text{Ext}^1(\mathcal{O}, \mathcal{O})^m=0$ because $\text{Ext}^1(\mathcal{O}, \mathcal{O})=0$ and thus part (c) follows.
\end{proof}

\subsubsection{} The following is the main result which implies Theorem \ref{halln}.

\begin{prop}\label{eq} Let $s\in \{a, b\}$. Then \[\text{DT}_{\pi}*1_{\textbf{T}_s}=\text{DT}_{\pi}^s*1_{\textbf{T}_s}*\text{PT}_{\pi}^s.\]
\end{prop}

Before we prove Proposition \ref{eq}, we need to introduce a few more notations. Fix $\pi$ a partition. Let $\textbf{A}$ be the abelian subcategory of $\text{Coh}_{\leq 1}(W)$ generated by sheaves $F$ as in \eqref{eq123}. It has torsion pairs $\left(\widetilde{\textbf{T}_s}, \widetilde{\textbf{F}_s}\right)$ constructed using the slope function $\mu$. We define the analogous stacks from Definition \eqref{deff} for the torsion pairs $\left(\widetilde{\textbf{T}_s}, \widetilde{\textbf{F}_s}\right)$, for example \begin{itemize}
    \item $\widetilde{\text{PT}^s_{\beta, n}}\subset \mathfrak{M}_{\beta, n}$ is the the locus of complexes $\left[\OO_{W}\xrightarrow{s} F\right]$ with $[F]=(\beta, n)\in N_{\leq 1}(W)$, $F\in \widetilde{\textbf{F}_s}$, and $\text{coker}(s)\in \widetilde{\textbf{T}_s}$,
    \item for $I\subset S$ an interval, $\widetilde{\text{SS}(I)}$ is the subcategory of $\text{SS}(I)$ of sheaves in $\textbf{A}$ and $1_{\widetilde{\text{SS}(I)}}$ is the sublocus of $\mathfrak{M}(\mathcal{O})$ represented by sheaves in $\widetilde{\text{SS}(I)}$
\end{itemize}
etc. Observe that $\widetilde{\textbf{T}_s}=\textbf{T}_s$ for $s\in\{a, b\}$.

\begin{prop}\label{iden} The following identities are true in the Hall algebra $H_{\text{sc},\Phi}$ for any $s\in \{a, b\}$:

(a) $1_{\widetilde{\textbf{T}_s}}*1_{\widetilde{\text{SS}[\tau,s)}}=1_{\widetilde{\text{SS}(\geq \tau)}}$ for $\tau<s$.

(b) $1^{\OO}_{\widetilde{\textbf{T}_s}}*1^{\OO}_{\widetilde{\text{SS}[\tau,s)}}=
1^{\OO}_{\widetilde{\text{SS}(\geq \tau)}}$ for $\tau<s$.

(c) $\widetilde{\text{DT}^s_{\pi}}*
1_{\widetilde{\textbf{T}_s}}=
1^{\OO}_{\widetilde{\textbf{T}_s}}.$

(d) $\widetilde{\text{PT}^s_\pi}*
1_{\widetilde{\text{SS}[\tau,s)}}-
1^{\OO}_{\widetilde{\text{SS}[\tau,s)}}\to 0$ as $\tau\to-\infty$. 

(e) $\widetilde{\text{DT}_\pi}*
1_{\widetilde{\text{SS}(\geq \tau)}}-1^{\OO}_{\widetilde{\text{SS}(\geq \tau)}}\to 0$ as $\tau\to-\infty$.
\end{prop}

\begin{proof}
The analogous identities for the category $\text{Coh}_{\leq 1}(W)$ with torsion pair $(\textbf{T}_s, \textbf{F}_s)$ were proved by Bridgeland \cite{b} for $s=b$ and Bryan-Steinberg \cite{bs} for $s=a$. These identities are proved in the same way.
Parts (a) and (b) use that $\widetilde{\textbf{T}_s}$ and $\widetilde{\text{SS}[\tau, s)}$ are a torsion pair for $\widetilde{\text{SS}(\geq \tau)}$, see Lemma 55 and Lemma 69 for (a), Lemma 57 for (b) in loc. cit. For $s=a$, see Lemma 59 for (c), Lemma 73 for (d) in loc. cit, and for (e) see Lemma 71 in loc. cit. For $s=b$, see \cite[Lemma 4.3]{b} for part (c) and \cite[Equation 39]{b} for part (d). 
\end{proof}

\begin{proof}[Proof of Proposition \ref{eq}]
We show that for $s\in \{a, b\}$, there are identities \begin{equation}\label{a1}
    \widetilde{\text{DT}_{\pi}}*1_{\widetilde{\textbf{T}_s}}=\widetilde{\text{DT}_{\pi}^s}*1_{\widetilde{\textbf{T}_s}}*\widetilde{\text{PT}_{\pi}^s}.\end{equation} The proof is the same as for the analogous identities for the full category $\text{Coh}_{\leq 1}(W)$, see \cite[Proposition 75]{bs} for $s=a$ and \cite[Proposition 6.5]{b} for $s=b$. 
The restriction of \eqref{a1} over $\mathfrak{N}_\pi$ gives the desired identities.
We abuse notation and write $\widetilde{\text{DT}_\pi}$ instead of $\left[\widetilde{\text{DT}_\pi} \to\mathfrak{M}\right]\in H_{\Phi}$ etc. in order to simplify the notation.

By Proposition \ref{iden}, part (e), we have that
\[\lim_{\tau\to-\infty}
\left(\widetilde{\text{DT}_\pi}*1_{\widetilde{\text{SS}(\geq \tau)}}-
1^{\OO}_{\widetilde{\text{SS}(\geq \tau)}}\right)\to 0.\] Using Proposition \ref{iden} part (a) and (b), this can be rewritten as
\begin{equation}\label{last}
    \lim_{\tau\to-\infty}
\left(\widetilde{\text{DT}_\pi}*
1_{\widetilde{\text{SS}(\geq s)}}*1_{\widetilde{\text{SS}[\tau,s)}}-1^{\OO}_{\widetilde{\text{SS}(\geq s)}}*1^{\OO}_{\widetilde{\text{SS}[\tau,s)}}\right)\to 0.\end{equation}
Using Proposition \ref{iden}, part (d), we also have that
\[\lim_{\tau\to-\infty}\left(\widetilde{\text{PT}^s_\pi}*1_{\widetilde{\text{SS}[\tau,s)}}-
1^{\OO}_{\widetilde{\text{SS}[\tau, s)}}\right)\to 0.\]
Multiply the above relation on the left by $1^{\OO}_{\widetilde{\text{SS}(\geq s)}}$, who by Proposition \ref{iden}, part (c), equals $\widetilde{\text{DT}^s_\pi}*1_{\widetilde{\text{SS}(\geq s)}}$ to obtain that
\[ \lim_{\tau\to-\infty}
\left(\widetilde{\text{DT}^s_\pi}*1_{\widetilde{\text{SS}(\geq s)}}*
\widetilde{\text{PT}^s_\pi}*
1_{\widetilde{\text{SS}[\tau,s)}}-
1^{\OO}_{\widetilde{\text{SS}(\geq s)}}*1^{\OO}_{\widetilde{\text{SS}[\tau, s)}}\right)\to 0.\] 
Combining this relation with \eqref{last}, we obtain that
\[ \lim_{\tau\to-\infty}\left(\widetilde{\text{DT}^s_\pi}*1_{\widetilde{\text{SS}(\geq s)}}*
\widetilde{\text{PT}^s_\pi}-
\widetilde{\text{DT}_\pi}*1_{\widetilde{\text{SS}(\geq s)}}\right)*1_{\widetilde{\text{SS}[\tau, s)}}\to 0.\]
The element $1_{\widetilde{\text{SS}[\tau,s)}}$ is invertible, see for the example the proof of \cite[Lemma 5.2]{b}, and thus we obtain the desired statement.
\end{proof}
 The following is \cite[Proposition 76 and Corollary 77]{bs}, and it is a corollary of Joyce's no-pole Theorem \cite[Theorem 3.11]{js}, \cite[Theorem 6.3]{b}.
 
\begin{prop}\label{nopole}
Let $\tau\in S$ be a slope. 
There exists $\nu_{\tau}=(\mathbb{L}-1)\varepsilon(\tau)\in H_{\text{reg}, \Phi}$ such that $$1_{\text{SS}(\tau)}=\exp(\varepsilon(\tau))\in H_{\Phi}.$$
Further, $1_{\text{SS}(\tau)}\in H_{\Phi}$ is invertible, the
automorphism $\text{Ad}_{1_{\text{SS}(\tau)}}:H_{\Phi}\to H_{\Phi}$ preserves
regular elements, and the induced Poisson automorphism of
$H_{\text{sc}}$ is given by $$\text{Ad}_{1_{\text{SS}(\tau)}}=\exp(\{\nu_{\tau},-
\}).$$
\end{prop}


\begin{proof}[Proof of Theorem $3.4$]
The proof follows \cite[Proof of Theorem 6]{bs}.
Use Theorem \ref{poi},
Propositions \ref{inte} and \ref{nopole}, and Proposition \ref{eq} for $s=b$ to obtain that
$$\Psi\left([\text{DT}_{\pi}\to\mathfrak{M}]\right)=
\Psi\left([\text{DT}_{\text{tors}}\to\mathfrak{M}]\right)(-1)^{d_\pi}\text{PT}_{\pi}(-q, z).$$ 
Further, for the trivial partition profile, we obtain that
$$\Psi\left([\text{DT}_{0}\to\mathfrak{M}]\right)=\Psi\left([\text{DT}_{\text{tors}}\to\mathfrak{M}]\right)\text{PT}_{0}(-q, z).$$
Use Theorem \ref{poi},
Propositions \ref{inte} and \ref{nopole}, and Proposition \ref{eq} for $s=a$ to obtain that
$$\Psi\left([\text{DT}_{\pi}\to\mathfrak{M}]\right)=\Psi\left([\text{DT}_{0}\to\mathfrak{M}]\right)(-1)^{d_\pi}\text{BS}_{\pi}(-q, z).$$ Putting together these equalities, we obtain that
$$\text{BS}_{\pi}(q, z)=\frac{\text{PT}_{\pi}(q, z)}{\text{PT}_{0}(q, z)}.$$
\end{proof}

\end{document}